\documentclass[openright,11pt]{article}

\usepackage{graphicx}
\usepackage[latin1]{inputenc}
\usepackage[english]{babel}
\usepackage{here}
\usepackage{amsmath,amssymb,amscd}
\usepackage{amsthm}
\usepackage{mathrsfs}
\usepackage{enumerate}
\usepackage[colorlinks=true,linkcolor=blue,urlcolor=blue]{hyperref}

\usepackage[flushmargin]{footmisc}
\newcommand\blfootnote[1]{%
	\begingroup
	\renewcommand\thefootnote{}\footnote{#1}%
	\addtocounter{footnote}{-1}%
	\endgroup}

\theoremstyle{plain}
\newtheorem{theorem}{Theorem}[section]
\newtheorem{proposition}[theorem]{Proposition}

\newtheorem{lemma}[theorem]{Lemma}
\newtheorem{corollary}[theorem]{Corollary}

\newtheorem{fact}{Fact}[subsection]

\theoremstyle{definition}
\newtheorem{definition}[theorem]{Definition}
\newtheorem{remark}[theorem]{Remark}
\newtheorem{example}[theorem]{Example}
\newtheorem{question}[theorem]{Question}
\newtheorem{problem}[theorem]{Problem}
\newtheorem{problems}[theorem]{Problems}

\newcommand{\NN}{\mathbb{N}}

\newcommand{\QQ}{\mathbb{Q}}
\newcommand{\RR}{\mathbb{R}}
\newcommand{\CC}{\mathbb{C}}
\newcommand{\KK}{\mathbb{K}}
\newcommand{\TT}{\mathbb{T}}

\newcommand{\Ac}{\mathcal{A}}
\newcommand{\AP}{\mathcal{AP}}

\newcommand{\Ci}{\mathscr{C}}
\newcommand{\Del}{\Delta}
\newcommand{\Ec}{\mathcal{E}}
\newcommand{\Fc}{\mathcal{F}}

\newcommand{\Ic}{\mathcal{I}}
\newcommand{\IP}{\mathcal{IP}}
\newcommand{\Lc}{\mathcal{L}}
\newcommand{\Nc}{\mathcal{N}}
\newcommand{\Part}{\mathscr{P}}

\newcommand{\Sc}{\mathcal{S}}
\newcommand{\Tc}{\mathcal{T}}
\newcommand{\TS}{\mathcal{TS}}
\newcommand{\Uc}{\mathcal{U}}

\newcommand{\Per}{\operatorname{Per}}
\newcommand{\URec}{\operatorname{URec}}
\newcommand{\FRec}{\operatorname{FRec}}
\newcommand{\UFRec}{\operatorname{UFRec}}
\newcommand{\RRec}{\operatorname{RRec}}
\newcommand{\Rec}{\textup{Rec}}

\newcommand{\FHC}{\operatorname{FHC}}
\newcommand{\UFHC}{\operatorname{UFHC}}
\newcommand{\RHC}{\operatorname{RHC}}
\newcommand{\HC}{\textup{HC}}

\newcommand{\Com}{\mathcal{C}}

\newcommand{\cl}{\overline}

\newcommand{\orb}{\textup{Orb}}
\newcommand{\lspan}{\textup{span}}
\newcommand{\res}{\arrowvert}
\newcommand{\eps}{\varepsilon}
\newcommand{\del}{\delta}
\newcommand{\til}{\widetilde} 
\newcommand{\diam}{\textup{diam}}

\newcommand{\ep}[2]{\langle #1 , #2 \rangle}

\newcommand{\Bdsup}{\overline{\operatorname{Bd}}}
\newcommand{\Bdinf}{\underline{\textup{Bd}}}
\newcommand{\dsup}{\overline{\operatorname{dens}}}
\newcommand{\dinf}{\underline{\operatorname{dens}}}
\newcommand{\BDsup}{\overline{\mathcal{BD}}}
\newcommand{\BDinf}{\underline{\mathcal{BD}}}
\newcommand{\Dsup}{\overline{\mathcal{D}}}
\newcommand{\Dinf}{\underline{\mathcal{D}}}

\begin{document}

\begin{center}
	\begin{LARGE}
		{\bf Questions in linear recurrence: From the $T\oplus T$-problem to lineability}
	\end{LARGE}
\end{center}

\vspace*{-0.4cm}

\begin{center}
	\begin{large}
		by
	\end{large}
\end{center}

\vspace*{-0.4cm}

\begin{center}
	\begin{large}
		Sophie Grivaux, Antoni L\'opez-Mart\'inez \& Alfred Peris\blfootnote{\textbf{Journal-ref (Part I)}: Analysis and Mathematical Physics, Volume 15, article number 1, (2025).\\ \textbf{DOI}: https://doi.org/10.1007/s13324-024-00999-8\\ \textbf{Journal-ref (Part II)}: Banach Journal of Mathematical Analysis, Volume 19, article number 61, (2025).\\ \textbf{DOI}: https://doi.org/10.1007/s43037-025-00448-z}
	\end{large}
\end{center}

\vspace*{-0.2cm}

\begin{abstract}
	We study, for a continuous linear operator $T$ on an F-space $X$, when the direct sum operator $T\oplus T$ is recurrent on $X\oplus X$. In particular: we establish, for recurrence, the analogous notion to that of (topological) {\em weak-mixing} for transitivity/hypercyclicity, namely {\em quasi-rigidity}; and we construct a recurrent but not quasi-rigid operator on each infinite-dimensional Banach space, solving the {\em $T\oplus T$-recurrence problem} in the negative way. The quasi-rigidity notion is closely related to the {\em dense lineability} of the set of recurrent vectors, and using similar conditions we study the lineability and dense lineability properties for the set of $\Fc$-recurrent vectors. \textbf{This document has been split into two already published papers (see \cite{GrivauxLoPe2025_AMP_questions-I} and \cite{GrivauxLoPe2025_BJMA_questions-II})}.
\end{abstract}

\vspace*{-0.2cm}

\section{Introduction}

This paper focuses on some aspects of {\em Linear Dynamics} and we have the following general setting: a ({\em real} or {\em complex}) {\em linear dynamical system} $(X,T)$ is a pair formed by a continuous linear operator $T:X\longrightarrow X$ on a (real or complex) separable infinite-dimensional F-space (i.e.\ $X$ is a completely metrizable topological vector space). We will denote by $\Lc(X)$ the {\em set of continuous linear operators} acting on such a space $X$, and $\KK$ will stand for either the real or complex field, $\RR$ or $\CC$.\\[-5pt]

Given a linear dynamical system $(X,T)$, the $T$-orbit of a vector $x \in X$ is the set
\[
\orb(x,T) := \{ x, Tx, T^2x, T^3x, ...\} = \{ T^nx : n \in \NN_0 \},
\]
where $\NN_0:=\NN\cup\{0\}$. We say that a vector $x \in X$ is {\em hypercyclic} for $T$ if $\orb(x,T)$ is a dense set in $X$; and that the operator $T$ is {\em hypercyclic} if it admits a hypercyclic vector. The so-called {\em Birkhoff transitivity theorem} (see for instance \cite[Theorem~1.2]{BaMa2009_book}) shows the equivalence between hypercyclicity and the notion of  topological transitivity: a system $(X,T)$ is called {\em topologically transitive} if for each pair of non-empty open subsets $U,V \subset X$ one can find some (and hence infinitely many) $n \in \NN_0$ such that $T^n(U) \cap V \neq \varnothing$. It follows that the set $\HC(T)$ of hypercyclic vectors for $T$ is a dense $G_{\delta}$ subset of $X$ as soon as $T$ is hypercyclic.\\[-5pt]

Hypercyclicity has been, historically, the main and most studied notion in Linear Dynamics. For instance, the following long-standing problem, which we call the {\em $T\oplus T$-hypercyclicity problem}, was posed in 1992 by D. Herrero \cite{Herrero1991}:

\begin{question}[\textbf{The $T\oplus T$-hypercyclicity problem}]\label{Q:T+T-hyp}
	Let $T$ be a hypercyclic operator on the F-space $X$. Is the operator $T\oplus T$ acting on the direct sum $X\oplus X$ hypercyclic?
\end{question}

Recall that an operator $T$ is said to be (topologically) {\em weakly-mixing} if and only if $T\oplus T$ is transitive, so the question above asks if there exists any transitive but not weakly-mixing operator. The active research on hypercyclicity yielded many equivalent reformulations of Question~\ref{Q:T+T-hyp}, and the approach which will interest us most is the one taken in 1999 by B\`es and the third author of this paper. Here is the main result they obtained in \cite{BesPe1999}:
\begin{enumerate}[--]
	\item {\em A continuous linear operator $T \in \Lc(X)$ is weakly-mixing if and only if it satisfies the so-called Hypercyclicity Criterion}.
\end{enumerate}
In other words, $T\oplus T \in \Lc(X\oplus X)$ is a hypercyclic operator if and only if the following holds: {\em there exist two dense subsets $X_0,Y_0 \subset X$, an increasing sequence of positive integers $(n_k)_{k\in\NN}$, and a family of (not necessarily continuous) mappings $S_{n_k}:Y_0\longrightarrow X$ such that
\begin{enumerate}[{\em(i)}]
	\item $T^{n_k}x \to 0$ for each $x \in X_0$;
	
	\item $S_{n_k}y \to 0$ for each $y \in Y_0$;
	
	\item $T^{n_k}S_{n_k}y \to y$ for each $y \in Y_0$. 
\end{enumerate}}
Question~\ref{Q:T+T-hyp} was finally answered negatively in 2006 by De La Rosa and Read, when they exhibited a hypercyclic operator on a particular Banach space whose direct sum with itself is not hypercyclic (see \cite{DeRead2009}). The techniques used there were later refined by Bayart and Matheron in order to construct hypercyclic but not weakly-mixing operators, on each Banach space admitting a normalized unconditional basis whose associated forward shift is continuous, as for instance on the classical $c_0(\NN)$ and $\ell^p(\NN)$ spaces, $1\leq p<\infty$ (see \cite{BaMa2007nwm,BaMa2009nwm}).\\[-5pt]

The aim of this paper is to study the previous problem, and to develop the corresponding theory, for the notion of {\em recurrence}: we say that a vector $x \in X$ is {\em recurrent} for $T$ if $x$ belongs to the closure of its forward orbit $\cl{\orb(Tx,T)}$; and that the operator $T$ is {\em recurrent} if its set $\Rec(T)$ of recurrent vectors is {\em dense} in $X$. Note that, in order to say that an operator $T$ has a {\em recurrent behaviour} it is not enough to assume that $\Rec(T)\neq\varnothing$ since the zero-vector is recurrent for every linear map. It is shown in \cite[Proposition~2.1]{CoMaPa2014} that the notion of recurrence coincides with that of {\em topological recurrence}, i.e.\ the property that for each non-empty open subset $U \subset X$ one can find some (and hence infinitely many) $n \in \NN$ such that $T^n(U) \cap U \neq \varnothing$. In addition, and as in the hypercyclicity case, when $T$ is recurrent its set $\Rec(T)$ of recurrent vectors is a dense $G_{\delta}$ subset of $X$.\\[-5pt]

Recurrence is one of the oldest and most studied concepts in the {\em Topological Dynamics} area of knowledge (see \cite{Furstenberg1981_book,Banks1999,KwiLiOpYe2017}), but in Linear Dynamics its systematic study began with the 2014 paper of Costakis, Manoussos and Parissis \cite{CoMaPa2014}. As one can naturally expect, many questions in {\em linear recurrence} remain open and here we will be particularly interested in the so-called {\em $T\oplus T$-recurrence problem}, which was stated in \cite[Question~9.6]{CoMaPa2014}:

\begin{question}[\textbf{The $T\oplus T$-recurrence problem}]\label{Q:T+T-rec}
	Let $T$ be a recurrent operator on the F-space $X$. Is the operator $T\oplus T$ acting on the direct sum $X\oplus X$ recurrent?
\end{question}

This question has been recently restated as an open problem in the 2021 paper \cite{ChenKosVe2021} for C$_0$-semigroups of operators. In the first part of this work we identify the systems $(X,T)$ for which the recurrence of $T$ implies that of $T\oplus T$. This characterization is given in terms of {\em quasi-rigidity} (see Definition~\ref{Def:qr}), which will be for recurrence, the analogous property to that of {\em weak-mixing} for hypercyclicity (see Theorem~\ref{The:Juan}). We then answer Question~\ref{Q:T+T-rec} negatively, by constructing some recurrent but not quasi-rigid operators (see Section~\ref{Sec:3n-qr}).

The second part of this paper is dedicated to the study of some properties related to the notion of lineability: recall that a subset $Y$ of an F-space $X$ is called ({\em dense}) {\em lineable} if $Y\cup\{0\}$ contains a (dense) infinite-dimensional vector space. A well-known result (due to Herrero and Bourdon) states that the set $\HC(T)$ is {\em dense lineable} for every hypercyclic operator $T$ (see \cite[Theorem 2.55]{GrPe2011_book}); moreover, one easily observes that if $T$ is a quasi-rigid operator, then $\Rec(T)$ is also {\em dense lineable} (see Proposition~\ref{Pro:qr->dense.lin}). Here we generalize these results to a rather general class of Furstenberg families $\Fc$: first we show that it makes sense to study quasi-rigidity from the $\Fc$-recurrence point of view (see Proposition~\ref{Pro:qr.filter}); and then we check that, for an $\Fc$-recurrent operator $T$, the set $\Fc\Rec(T)$ of $\Fc$-recurrent vectors is always {\em lineable} and usually even {\em dense lineable} (see Theorems~\ref{The:lineability} and \ref{The:dense.lineability}). As a consequence we obtain the {\em Herrero-Bourdon theorem} for $\Fc$-hypercyclicity (see Subsection~\ref{SubSec:5.3dense.lin.F-hyp}).\\[-5pt]

The paper is organized as follows. Section~\ref{Sec:2qr} is devoted to define and study {\em quasi-rigidity}, which will be the analogous recurrence property to that of {\em weak-mixing} for hypercyclicity. We show in Section~\ref{Sec:3n-qr} that there exist recurrent but not quasi-rigid operators, answering Question~\ref{Q:T+T-rec} in a negative way and solving the {\em $T\oplus T$-recurrence problem}. In Section~\ref{Sec:4F-rec} we recall the definition of $\Fc$-recurrence, showing that quasi-rigidity is a particular case of such a concept, and studying both the weakest and strongest possible $\Fc$-recurrence notions. The {\em lineability} and {\em dense lineability} properties, for the set $\Fc\Rec(T)$ of $\Fc$-recurrent vectors, are studied in Section~\ref{Sec:5lin+dense.lin}. We finally gather, in Section~\ref{Sec:6open}, some left open problems.\\[-5pt]

We refer the reader to the textbooks \cite{BaMa2009_book,GrPe2011_book} for any unexplained but standard notion about hypercyclicity, or more generally, about Linear Dynamics. Along the paper we use the environment ``\textbf{Question}'' to state the different problems that we solve, and we use the word ``\textbf{Problem(s)}'' to state the left open problems.

\section{Quasi-rigid dynamical systems}\label{Sec:2qr}

In this section we introduce and study the concept of {\em quasi-rigidity}, which is naturally defined in the broader framework of {\em Topological Dynamics} (i.e.\ without assuming linearity): a pair $(X,T)$ is called a {\em dynamical system} if $T$ is a continuous self-map on a Hausdorff topological space $X$. An important class of dynamical systems, which we will repeatedly use, is the family of Polish systems: a pair $(X,T)$ is said to be a {\em Polish dynamical system} whenever $T$ is a continuous self-map of a separable completely metrizable space $X$.

\begin{definition}
	Given any dynamical system $(X,T)$ and any $N \in \NN$ we will denote by $T_{(N)}:X^N\longrightarrow X^N$ the {\em $N$-fold direct product} of $T$ with itself, i.e.\ $(X^N,T_{(N)})$ is the dynamical system
	\[
	T_{(N)} := \underbrace{T\times\cdots\times T}_{N} : \underbrace{X\times\cdots\times X}_{N} \longrightarrow \underbrace{X\times\cdots\times X}_{N},
	\]
	where $X^N:=\underbrace{X\times\cdots\times X}_{N}$ is the {\em $N$-fold product} of $X$ and $T_{(N)}(x_1,...,x_N) := (Tx_1,...,Tx_N)$.
\end{definition}
If $(X,T)$ is linear, then $(X^N,T_{(N)})$ will refer to the {\em $N$-fold direct sum} linear dynamical system
\[
T_{(N)} := \underbrace{T\oplus\cdots\oplus T}_{N} : \underbrace{X\oplus\cdots\oplus X}_{N} \longrightarrow \underbrace{X\oplus\cdots\oplus X}_{N},
\]
and in this case $X^N:=\underbrace{X\oplus\cdots\oplus X}_{N}$ is the {\em $N$-fold direct sum} of $X$.

\subsection{Quasi-rigidity: definition and equivalences}\label{SubSec:2.1qr}

By a well-known theorem of Furstenberg, once a dynamical system $(X,T)$ is weakly-mixing (i.e.\ if $T\times T$ is topologically transitive) then so is every $N$-fold direct product system $(X^N,T_{(N)})$. See \cite[Theorem~1.51]{GrPe2011_book}. Hence, in order to completely answer Question~\ref{Q:T+T-rec} we should first study the recurrent dynamical systems $(X,T)$ for which every $N$-fold direct product is again recurrent. A first attempt would be to rely on the well-known and classical notion of {\em rigidity}: a dynamical system $(X,T)$ is said to be {\em rigid} if there exists a strictly increasing sequence $(n_k)_{k\in\NN} \in \NN^{\NN}$ such that $T^{n_k}x \to x$ for every $x \in X$. Rigidity has been studied in different contexts such as measure theoretic recurrence \cite{FursWeiss1977}, dynamical systems on topological spaces \cite{GlasMaon1989} and also for linear systems \cite{EisGri2011,CoMaPa2014}. It is a really strong form of recurrence which implies that $X^N=\Rec(T_{(N)})$ for all $N \in \NN$, so this is not the exact notion we are looking for. Nonetheless, it motivates the definition:

\begin{definition}[\textbf{Quasi-rigidity}]\label{Def:qr}
	A dynamical system $(X,T)$ will be called:
	\begin{enumerate}[--]
		\item {\em quasi-rigid} if there exist a strictly increasing sequence $(n_k)_{k \in \NN} \in \NN^{\NN}$ and a dense subset $Y$ of $X$ such that $T^{n_k}x \to x$, as $k\to\infty$, for every $x \in Y$. In this case we will say that $T$ is {\em quasi-rigid with respect to $(n_k)_{k \in \NN}$}.
		
		\item {\em topologically quasi-rigid} if there exists $(n_k)_{k \in \NN} \in \NN^{\NN}$ strictly increasing and such that: for every non-empty open subset $U \subset X$ there is $k_U \in \NN$ with $T^{n_k}(U) \cap U \neq \varnothing$ for all $k\geq k_U$. In this case we will say that $T$ is {\em topologically quasi-rigid with respect to $(n_k)_{k \in \NN}$}.
	\end{enumerate}
\end{definition}

Our aim is to show that, given a dynamical system $(X,T)$, the previous notions characterize the recurrent behaviour of every $N$-fold direct product system $(X^N,T_{(N)})$ under very weak assumptions (dynamically speaking) on the underlying space $X$. The following result shows that, when the underlying space $X$ is second-countable, {\em topological quasi-rigidity} is for recurrence the analogous property to that of {\em weak-mixing} for transitivity:

\begin{lemma}\label{Lem:t.qr.2AN}
	Let $(X,T)$ be a dynamical system on a second-countable space $X$. The following are equivalent:
	\begin{enumerate}[{\em(i)}]
		\item $T$ is topologically quasi-rigid;
		
		\item $T_{(N)}$ is topologically recurrent for every $N \in \NN$.
	\end{enumerate}  
\end{lemma}
\begin{proof}
	(i) $\Rightarrow$ (ii): suppose that $T$ is topologically quasi-rigid with respect to $(n_k)_{k\in\NN}$, let $N \in \NN$ and consider any finite sequence of non-empty open sets $U_1,...,U_N \subset X$. Using now the topological quasi-rigidity assumption we can find $k_0 \in \NN$ such that
	\[
	T^{n_k}(U_j) \cap U_j \neq \varnothing \quad \text{ for every } 1\leq j\leq N \text{ and all } k\geq k_0.
	\]
	Then, for the open set $U=U_1\times\cdots\times U_N \subset X^N$ we have that $T_{(N)}^{n_k}(U) \cap U \neq \varnothing$ for all $k\geq k_0$, which implies that $T_{(N)}$ is topologically recurrent since $U_1,...,U_N$ are arbitrary.\\[-5pt]
	
	(ii) $\Rightarrow$ (i): assume that $T_{(N)}$ is topologically recurrent for every $N \in \NN$ and let $(U_s)_{s \in\NN}$ be a countable base of (non-empty) open sets for $X$. We can construct recursively an increasing sequence $(n_k)_{k\in\NN}$ such that $T^{n_k}(U_s) \cap U_s \neq \varnothing$ for every $s,k \in \NN$ with $1\leq s\leq k$: given any $k \in \NN$, the set $U:=U_1\times\cdots\times U_k$ is a non-empty open subset of $X^k$ and by the topological recurrence of the map $T_{(k)}$ we can pick $n_k \in \NN$ sufficiently large with
	\[
	T_{(k)}^{n_k}(U) \cap U \neq \varnothing \quad \text{ and hence } \quad T^{n_k}(U_s) \cap U_s \neq \varnothing \quad \text{ for all } 1\leq s \leq k.
	\]
	It is easy to check that $T$ is topologically quasi-rigid with respect to $(n_k)_{k\in\NN}$.
\end{proof}

As one would expect, if the underlying space is completely metrizable then we can identify the ``pointwise'' {\em quasi-rigidity} notion with that of {\em topological quasi-rigidity}:

\begin{proposition}\label{Pro:qr.complete}
	Let $(X,T)$ be a Polish dynamical system. The following are equivalent:
	\begin{enumerate}[{\em(i)}]
		\item $T$ is quasi-rigid;
		
		\item $T$ is topologically quasi-rigid.
	\end{enumerate}  
\end{proposition}
\begin{proof}
	The implication (i) $\Rightarrow$ (ii) is straightforward: assume that there is a dense subset $Y \subset X$ and an increasing sequence $(n_k)_{k\in\NN}$ such that $T^{n_k}x \to x$ for every $x \in Y$. For each non-empty open subset $U \subset X$ we can select $x \in U \cap Y$, and then
	\[
	T^{n_k}x \in T^{n_k}(U) \cap U \quad \text{ for every sufficiently large } k \in \NN.
	\]
	Hence $T$ is topologically quasi-rigid with respect to $(n_k)_{k\in\NN}$. Let us prove (ii) $\Rightarrow$ (i): assume that $d(\cdot,\cdot)$ is a metric defining the complete topology of $X$, let $(U_s)_{s\in\NN}$ be a countable base of open sets for $X$ and set $V_{s,s}:=U_s$ for each $s \in \NN$. By (ii) we can find $n_1 \in \NN$ such that
	\[
	T^{n_1}(U_1) \cap U_1 = T^{n_1}(V_{1,1}) \cap V_{1,1} \neq \varnothing.
	\]
	The continuity of $T$ ensures the existence of a non-empty open subset $V_{1,2} \subset X$ of diameter (with respect to $d$) less than $\frac{1}{2^2}$ for which $\cl{V_{1,2}} \subset V_{1,1}$ and $T^{n_1}\left(\cl{V_{1,2}}\right) \subset V_{1,1}$. Suppose now that for some $k \in \NN$ we have already constructed:
	\begin{enumerate}[--]
		\item finite sequences $(V_{s,j})_{s\leq j\leq k}$ of open subsets of $X$, for each $1\leq s\leq k-1$;
		
		\item and a finite increasing sequence of positive integers $(n_j)_{1\leq j\leq k-1}$;
	\end{enumerate}
	with the properties that $V_{s,j}$ has $d$-diameter less than $\frac{1}{2^j}$ when $s<j$, and also that
	\[
	\cl{V_{s,j}} \subset V_{s,j-1} \quad \text{ and } \quad T^{n_{j-1}}\left(\cl{V_{s,j}}\right) \subset V_{s,j-1} \quad \text{ for all } 1\leq s < j \leq k.
	\]
	Then, considering the open subsets $V_{1,k}, V_{2,k}, ..., V_{k,k} \subset X$ and using again (ii) we can select a positive integer $n_k \in \NN$ with $n_k>n_{k-1}$ for which $T^{n_k}(V_{s,k}) \cap V_{s,k} \neq \varnothing$ for all $1\leq s\leq k$. Again the continuity ot $T$ ensures the existence, for each $1\leq s\leq k$, of a non-empty open subset $V_{s,k+1} \subset X$ of $d$-diameter less than $\frac{1}{2^{k+1}}$ and such that
	\[
	\cl{V_{s,k+1}} \subset V_{s,k} \quad \text{ and } \quad T^{n_{k}}\left(\cl{V_{s,k+1}}\right) \subset V_{s,k} \quad \text{ for all } 1\leq s\leq k.
	\]
	A recursive argument gives us an increasing sequence of positive integers $(n_k)_{k\in\NN}$ and, for each $s \in \NN$, a sequence of non-empty open sets $(V_{s,k})_{k=s}^{\infty}$ such that each set $V_{s,k}$ has $d$-diameter less than $\frac{1}{2^k}$ when $s<k$, but also satisfying
	\[
	\cl{V_{s,k+1}}\subset V_{s,k} \quad \text{ and } \quad  T^{n_k}\left(\cl{V_{s,k+1}}\right)\subset V_{s,k} \quad \text{ for all } s,k \in \NN \text{ with } s\leq k.
	\]
	By the Cantor intersection theorem, for each $s\in\NN$ there is a unique vector $y_s\in X$ with
	\[
	\{y_s\} = \bigcap_{k\geq s} \cl{V_{s,k}} \subset V_{s,s} = U_s.
	\]
	We deduce that $Y:=\{ y_s: s \in \NN \}$ is a dense subset of $X$. Since for each $s \in \NN$ we have that $T^{n_k}y_s \in T^{n_k}\left(\cl{V_{s,k+1}}\right) \subset V_{s,k}$ for all $k > s$, we get that
	\[
	\limsup_{k\to\infty} d(T^{n_k}y_s,y_s) \leq \limsup_{k\to\infty} \left(\diam_d(V_{s,k})\right) \leq \limsup_{k\to\infty} \frac{1}{2^k} = 0,
	\]
	and hence $T$ is quasi-rigid with respect to the sequence $(n_k)_{k\in\NN}$.
\end{proof}

The proof of the implication (i) $\Rightarrow$ (ii) in Proposition~\ref{Pro:qr.complete} shows that every quasi-rigid map is topologically quasi-rigid even if $X$ is not a metrizable space. The converse fails in general (consider for instance \cite[Example 12.9]{GrPe2011_book}). Combining the previous results we get:

\begin{theorem}\label{The:Juan}
	Let $(X,T)$ be a Polish dynamical system. The following are equivalent:
	\begin{enumerate}[{\em(i)}]
		\item $T$ is quasi-rigid;
		
		\item $T$ is topologically quasi-rigid;
		
		\item $T_{(N)}$ is topologically recurrent for every $N \in \NN$;
		
		\item $T_{(N)}$ is recurrent for every $N \in \NN$.
	\end{enumerate}
\end{theorem}
\begin{proof} 
	The implications (i) $\Rightarrow$ (ii) $\Rightarrow$ (iii) (and also (i) $\Rightarrow$ (iv) $\Rightarrow$ (iii)) are trivial even if the space $X$ is not second-countable neither complete. When $X$ is second-countable we get the equivalence (ii) $\Leftrightarrow$ (iii) by Lemma~\ref{Lem:t.qr.2AN}. If we just assume completeness, then we have the equivalence (iii) $\Leftrightarrow$ (iv) by \cite[Proposition 2.1]{CoMaPa2014}. Finally, using the Polish hypothesis we get the equivalence (i) $\Leftrightarrow$ (ii) by Proposition~\ref{Pro:qr.complete}.
\end{proof}

\begin{remark}
	Even though we have worked in the general setting of {\em Polish dynamical systems}, it is worth mentioning that Theorem~\ref{The:Juan} holds true for every:
	\begin{enumerate}[--]
		\item {\em compact dynamical system}, i.e.\ when $T$ is a continuous self-map of a compact metrizable space $X$ (many references only study this class of systems, see for instance \cite{Banks1999,Furstenberg1981_book,KwiLiOpYe2017});
		
		\item continuous linear operator $T$ acting on any separable F-space $X$. In other words, the previous result  holds true for {\em linear dynamical systems} as defined in the Introduction.
	\end{enumerate}
	In the linear setting we have got the following relations between the already exposed concepts (we denote by ``\textit{HC}'' the \textit{Hypercyclicity Criterion}):
	\begin{equation*}
		\begin{CD}
			\text{{\em satisfying the HC}}  @>{\text{Theorem~\ref{The:Juan}}}>>  \text{{\em quasi-rigidity}} \\
			@VVV @VVV \\
			\text{{\em hypercyclicity}}  @>{\text{already known}}>> \text{{\em recurrence}}
		\end{CD}
	\end{equation*}$\ $\\
	The notion of {\em quasi-rigidity} is for recurrence, the analogous property to that of {\em satisfying the Hypercyclicity Criterion} (and hence to that of {\em weak-mixing}) for hypercyclicity (see \cite{BesPe1999}).
\end{remark}

\subsection{Quasi-rigid operators}

It is now time to discuss how the quasi-rigidity notion influences the structure of the set of recurrent vectors for a linear dynamical system $(X,T)$. We start by showing that quasi-rigidity implies the dense lineability of the set of recurrent vectors:

\begin{proposition}\label{Pro:qr->dense.lin}
	If $T:X\longrightarrow X$ is a quasi-rigid operator, then $\Rec(T)$ is dense lineable.
\end{proposition}
\begin{proof}
	Assume that $Y \subset X$ is a dense subset such that there exists a sequence $(n_k)_{k\in\NN}$ of integers with $T^{n_k}x \to x$ for all $x \in Y$. Then any $z \in Z := \lspan(Y)$ satisfies that $T^{n_k}z \to z$, so $Z \subset X$ is an infinite-dimensional dense vector space contained in $\Rec(T)$.
\end{proof}

\begin{remark}
	Let $T:X\longrightarrow X$ be a quasi-rigid operator with respect to a sequence $(n_k)_{k\in\NN}$ and consider $Y := \left\{ x \in X : T^{n_k}x \to x \text{ as } k\to\infty \right\}$. This is a dense subset of $X$ which can be either of first or second Baire category, depending on the operator $T$:\newpage
	\begin{enumerate}[(a)]
		\item If $T$ is a {\em rigid} operator, then $Y=X$. This is the case of the identity operator, but there also exist examples of rigid systems which are even (Devaney) chaotic (see \cite{EisGri2011}).
		
		\item If $X$ is a Banach space and $Y$ is a second category set, then $\sup_{k \in \NN} \|T^{n_k}\| < \infty$ by the Banach-Steinhaus theorem. Consider the operator $T:=\lambda B$ with $|\lambda|>1$, where $B$ is the (unilateral) {\em backward shift} on $X=c_0(\NN)$ or any $\ell^p(\NN)$ ($1\leq p<\infty$). This operator is weakly-mixing and hence quasi-rigid, but the set $Y$ as defined before has to be of first category since $x=(\frac{1}{n^2})_{n\in\NN} \in X$ has the property that $\|T^n x\| \to \infty$ as $n \to \infty$.
	\end{enumerate}
\end{remark}

Now we give some (usually fulfilled) sufficient conditions for a dynamical system to be quasi-rigid. Recall that a vector $x \in X$ is called {\em cyclic} for an operator $T:X\longrightarrow X$ if
\[
\lspan(\orb(x,T)) = \lspan\{T^nx : n \in \NN_0\} = \left\{ p(T)x : p \text{ polynomial} \right\},
\]
is a dense set in $X$; and an operator $T$ is called {\em cyclic} as soon as it admits a cyclic vector. Moreover, a point $x \in X$ is called {\em periodic} for a map $T$ whenever $T^px=x$ for some positive integer $p\in\NN$. See \cite{BaMa2009_book,GrPe2011_book} for more on cyclicity and periodicity.

\begin{proposition}\label{Pro:sufficient-quasi-rigid}
	Let $(X,T)$ be a dynamical system for which any of the following holds:
	\begin{enumerate}[{\em(a)}]
		\item the set $\Per(T)$ of periodic points for $T$ is dense in $X$;
		
		\item $T$ admits a recurrent vector $x \in \Rec(T)$ for which the following set is dense
		\[
		\{Sx : ST = TS, S:X\longrightarrow X \text{ continuous mapping}\};
		\]
	\end{enumerate}
	then $T$ is quasi-rigid. In particular, a continuous linear operator $T$ acting on a Hausdorff topological vector space $X$ is quasi-rigid whenever any of the following holds:
	\begin{enumerate}[--]
		\item $T$ has dense periodic vectors;
		
		\item $T$ admits a recurrent and cyclic vector;
		
		\item $T$ admits a hypercyclic vector.
	\end{enumerate}
\end{proposition}
\begin{proof}
	To check (a) consider $Y=\Per(T)$ and $(n_k)_{k\in\NN}=(k!)_{k\in\NN}$. For (b) consider
	\[
	Y=\{Sx : ST = TS, S:X\longrightarrow X \text{ continuous mapping}\}.
	\]
	Indeed, since $x$ is recurrent there exists $(n_k)_{k\in\NN}$ such that $T^{n_k}x\to x$, but then, given any point $y=Sx \in Y$ we get that $T^{n_k}(Sx) = S(T^{n_k}x) \to Sx$. Finally, recall that every recurrent and cyclic (and hence every hypercyclic) vector fulfills condition (b).
\end{proof}

Proposition~\ref{Pro:sufficient-quasi-rigid} together with \cite[Corollary~9.2]{CoMaPa2014} show that recurrence plus cyclicity imply quasi-rigidity when $X$ is a complex Banach space. This result extends to (real and complex) linear dynamical systems as defined in the Introduction via the following lemma, in which the concept of {\em complexification} is used (see Remark~\ref{Rem:complexification}). We write $\TT = \{ z \in \CC : |z|=1 \}$.

\begin{lemma}\label{Lem:poly-dense-range}
	Let $T$ be a topologically recurrent continuous linear operator acting on a (real or complex) Hausdorff topological vector space $X$. The following statements hold:
	\begin{enumerate}[{\em(a)}]
		\item For every (real or complex) number $\lambda \in \KK\setminus\TT$ the operator $(T-\lambda)$ has dense range.
		
		\item If $X$ is a real space, then for every $\lambda \in \CC\setminus\TT$ the complexified operator $(\til{T}-\lambda)$, acting on the complexification $\til{X}$ of $X$, has dense range.
	\end{enumerate}
	As a consequence, for any (real or complex) polynomial $p$ without unimodular roots, the operator $p(T)$ has dense range on the (real or complex) space $X$.
\end{lemma}\newpage
The proof of this lemma is based on \cite[Lemma 12.13]{GrPe2011_book}, a well-known result of Wengenroth for topologically transitive operators acting on Hausdorff topological vector spaces. We just show statement (a), since (b) follows exactly as in \cite[Exercise 12.2.6]{GrPe2011_book} ({\em Hint}: when $X$ is a real space define the {\em complexification} $\til{T}:\til{X}\longrightarrow \til{X}$ as it is done in Remark~\ref{Rem:complexification} for operators on F-spaces, and use the fact that for each non-empty open subset $U \subset X$ there are infinitely many positive integers $n \in \NN$ with $\til{T}^n(U+iU)\cap(U+iU)\neq\varnothing$).
\begin{proof}[Proof of Lemma~\ref{Lem:poly-dense-range}]
	Given $\lambda \in \KK\setminus\TT$ let $E:=\cl{(T-\lambda)(X)}$, which is a closed subspace of $X$, and suppose that $E\neq X$. The quotient space $X/E$ is then a (real or complex) Hausdorff topological vector space (since $E$ is closed). If $q:X\longrightarrow X/E$ is the quotient map we have that $q((T-\lambda)x)=0$ and $q(Tx) = \lambda q(x)$ for all $x\in X$. The operator $S$ on $X/E$ given by $S[x] = \lambda[x]$ is easily seen to be topologically recurrent since so is $T$.\\[-5pt]
	
	However, if $|\lambda|<1$ we can choose $[x] \in X/E\setminus\{0\}$ and a balanced $0$-neighbourhood $W$ such that $[x] \notin \cl{W}$. Since $\lambda^n[x] \to 0$ and $\lambda W \subset W$ there exists a neighbourhood $U$ of $[x]$ and a natural number $N \in \NN$ for which $U\cap W=\varnothing$ and $\lambda^nU \subset W$ for every $n\geq N$. Hence $S^n(U) \cap U = \varnothing$ for all $n\geq N$ so $S$ is not topologically recurrent, which is a contradiction. If $|\lambda|>1$ we can consider the map $S^{-1}:X/E\longrightarrow X/E$ for which $S^{-1}[x]=\lambda^{-1}[x]$. By the previous argument $S^{-1}$ is not topologically recurrent, but then $S$ is not topologically recurrent (see \cite[Section 6]{GriLo2023}), which is again a contradiction.\\[-5pt]
	
	Let now $p$ be any (real or complex) polynomial without unimodular roots: if $X$ is a complex space then $p(T)$ has dense range since, splitting the polynomial, $p(T)$ can be written as a composition of dense range operators by (a); and if $X$ is a real space then $p(\til{T})$ also has dense range in $\til{X}$ via the same argument as above but using (b). Finally, if $p$ has real coefficients, then $p(\til{T})=p(T)+ip(T)$ and $p(T)$ has dense range in the real space $X$.
\end{proof}

The proof of \cite[Corollary 9.2]{CoMaPa2014} is highly based on the following result of D. Herrero: {\em given an operator $T$ on a complex Banach space $X$, if $\sigma_p(T^*)$ has empty interior then the set of cyclic vectors for $T$ is dense, and hence it is a dense $G_{\delta}$-set} (see \cite{Herrero1979}). By using Lemma~\ref{Lem:poly-dense-range} we obtain this density result in the more general setting of operators on F-spaces:

\begin{corollary}[\textbf{Extension of \cite[Corollary 9.2]{CoMaPa2014}}]\label{Cor:Rec+Cyc}
	Let $(X,T)$ be a (real or complex) linear dynamical system. If $T$ is recurrent and cyclic then the set of cyclic vectors for $T$ is a dense $G_{\delta}$-set. In particular, if $T$ is recurrent and cyclic, then it is quasi-rigid.
\end{corollary}
\begin{proof}
	The set of cyclic vectors is always a $G_{\delta}$-set. To check its density let $x \in X$ be a cyclic vector for $T$ and pick a non-empty open set $U \subset X$. Then there is a (real or complex) polynomial $p$ such that $p(T)x \in U$. Splitting $p$ we have that
	\[
	\textstyle p(T) = \prod_{j=1}^N (T-\lambda_j) \quad \text{ for some } (\lambda_j)_{j=1}^N \subset \CC.
	\]
	Choose a small perturbation $(\mu_j)_{j=1}^N \subset \CC\setminus\TT$ of $(\lambda_j)_{j=1}^N \subset \CC$ with the property that the polynomial
	\[
	\textstyle \widehat{p}(T) := \prod_{j=1}^N (T-\mu_j) \quad \text{ is still such that } \widehat{p}(T)x \in U.
	\]
	Note that if $p$ has real coefficients, then $(\mu_i)_{i=1}^N \subset \CC\setminus\TT$ can be chosen such that $\widehat{p}$ has still real coefficients. Finally, $\widehat{p}(T)x$ is a cyclic vector for $T$ since
	\[
	\widehat{p}(T) \left( \{ q(T)x : q \text{ polynomial} \} \right) = \{ q(T)(\widehat{p}(T)x) : q \text{ polynomial} \}
	\]
	is a dense set by Lemma~\ref{Lem:poly-dense-range}. Then $\Rec(T)$ and the set of cyclic vectors for $T$ are both residual, so $T$ admits a recurrent and cyclic vector. Proposition~\ref{Pro:sufficient-quasi-rigid} ends the proof.
\end{proof}

\section{Existence of recurrent but not quasi-rigid operators}\label{Sec:3n-qr}

After having studied in Section~\ref{Sec:2qr} the dynamical systems $(X,T)$ for which all their $N$-fold direct sums $(X^N,T_{(N)})$ are recurrent, we are now ready to answer negatively Question~\ref{Q:T+T-rec} showing that there exist recurrent but not quasi-rigid operators. If for the moment we continue working in the (not necessarily linear) {\em dynamical systems} setting, the following are natural questions in view of the theory already developed in the previous section:
\begin{question}\label{Q:T+T->qr}
	Let $(X,T)$ be a recurrent dynamical system:
	\begin{enumerate}[(a)]
		\item Is $(X\times X,T\times T)$ again a recurrent dynamical system?
		
		\item Assume that $(X\times X,T\times T)$ is recurrent. Is then $(X,T)$ quasi-rigid?
	\end{enumerate}
\end{question}
If we change the ``recurrence'' assumption into that of ``topological transitivity'', and in (b) we also exchange ``quasi-rigid'' for ``weakly-mixing'', then the respective answers are:
\begin{enumerate}[(a)]
	\item \textbf{No}, and in particular:
	\begin{enumerate}[--]
		\item for non-linear systems consider any irrational rotation, see \cite[Example~1.43]{GrPe2011_book};
		
		\item for linear maps we already mentioned in the Introduction the references \cite{DeRead2009,BaMa2007nwm,BaMa2009nwm}.
	\end{enumerate}
	
	\item \textbf{Yes}, in both linear and non-linear cases $(X^N,T_{(N)})$ is transitive for every $N \in \NN$ if and only if $T\times T$ is transitive (see \cite[Theorem~1.51]{GrPe2011_book}). 
\end{enumerate}

However, the answer to both questions is \textbf{no} in the recurrence setting, and indeed: in \cite[Lemma~9 and Example~4]{Banks1999} it is shown that for each $N \in \NN$ there exists a \textbf{non-linear dynamical system} $f:X\longrightarrow X$ on a \textbf{compact metric space} $X$ for which
\[
(X^N,f_{(N)}) \text{ is recurrent }\quad \text{ while } \quad (X^{N+1},f_{(N+1)}) \text{ is not recurrent}.
\]
This construction is \textbf{highly non-linear} since each compact space $X$ is the disjoint union of $N+1$ {\em sub-shifts} of the well-known {\em shift on two symbols}. It turns out that the answer to Question~\ref{Q:T+T->qr} in the framework of Linear Dynamics is also negative. Let us state the main result of this section:

\begin{theorem}\label{The:N-counter-example}
	Let $X$ be any (real or complex) separable infinite-dimensional Banach space. For each $N \in \NN$ there exists an operator $T \in \Lc(X)$ such that
	\[
	T_{(N)}:X^N\longrightarrow X^N \text{ is recurrent, and even } \Rec(T_{(N)})=X^N,
	\]
	but for which the operator $T_{(N+1)}:X^{N+1}\longrightarrow X^{N+1}$ is not recurrent any more.
\end{theorem}

Considering $N=1$ we obtain examples of recurrent operators $T \in \Lc(X)$ for which $T\oplus T$ is not recurrent. Thus we get a negative answer to the {\em $T\oplus T$-recurrence problem}, i.e.\ we answer \cite[Question~9.6]{CoMaPa2014} and hence Question~\ref{Q:T+T-rec} in the negative.\\[-5pt]

Our proof of Theorem~\ref{The:N-counter-example} \textbf{relies heavily} on a construction of Aug\'e in \cite{Auge2012} of operators $T$ on (infinite-dimensional, separable) Banach spaces $X$ with \textbf{wild dynamics}: the two sets
\[
A_T = \left\{ x \in X : \lim_{n\to\infty} \|T^nx\| = \infty \right\} \quad \text{ and } \quad B_T = \left\{ x \in X : \liminf_{n\to\infty} \|T^nx-x\|=0 \right\}
\]
have non-empty interior and form a partition of $X$. Since $B_T=\Rec(T)$, these operators have plenty (but not a dense set) of recurrent vectors. In the remainder of this section we prove the \textbf{complex} version of Theorem~\ref{The:N-counter-example} by modifying the construction of \cite{Auge2012}. The \textbf{real} case follows in a really similar way by using the same arguments as in \cite[Section~3.2]{Auge2012}.

\subsection{Necessary prerequisites}

Assume that $X$ is a complex separable infinite-dimensional Banach space, denote by $X^*$ its {\em topological dual space} and by $B_{X^*}$ the {\em unit ball} of such a space $X^*$. Moreover, given $(x,x^*) \in X\times X^*$ we will denote by $\ep{x^*}{x} = x^*(x)$ the dual evaluation. As in \cite{Auge2012}, our operator will be built using a bounded biorthogonal  system of $X$. The following result is well-known (see \cite[Vol I, Section 1.f]{LindTzaf1977_book} or \cite{OvsPel}):

\begin{enumerate}[--]
	\item {\em Given a separable Banach space $X$ one can find $(e_k,e_k^*)_{k\in\NN} \subset X\times X^*$ such that:
		\begin{itemize}
			\item $\lspan\{ e_k : k \in \NN \}$ is dense in $X$;
			
			\item $\ep{e_k^*}{e_j} = \delta_{k,j}$ where $\delta_{k,j}=0$ if $k\neq j$ and $1$ if $k=j$;
			
			\item for each $k \in \NN$ we have that $\|e_k\| = 1$, and $K := \sup_{k\in\NN} \|e_k^*\|^* < \infty$.
		\end{itemize}}
\end{enumerate}

Once we have fixed a sequence $(e_k,e_k^*)_{k\in\NN} \subset X\times X^*$ with the previous properties we set
\[
c_{00} := \lspan\{ e_k : k \in \NN \}.
\]
Given $x \in X$ we will write $x_k := \ep{e_k^*}{x}$ for each $k\in\NN$, and we will repeatedly use that
\begin{equation}\label{eq:1K.norma.x}
	\|x_k e_k\| \leq K \|x\| \quad \text{ for each } k \in \NN,
\end{equation}
since
\[
\|x\| := \sup_{x^* \in B_{X^*}} |\ep{x^*}{x}| \geq \frac{|\ep{e_k^*}{x}|}{\|e_k^*\|} = \frac{|x_k|}{\|e_k^*\|} = \frac{\|x_k e_k\|}{\|e_k^*\|} \geq \frac{\|x_k e_k\|}{K},
\]\\
for each $k \in \NN$, and hence that
\begin{equation}\label{eq:2K.eps}
	\|x-y\|<\eps \ \text{ implies } \ |x_k-y_k|<K\eps \quad \text{ for all } x,y \in X \text{ and } k \in \NN.
\end{equation}

\begin{remark}
	Recall that, if the sequence $(e_k)_{k\in\NN}$ is not a Schauder basis of $X$ then a vector $x \in X$ cannot be written  in general as a convergent series $x = \sum_{k\in\NN} x_k e_k$. However, for the vectors in $c_{00} = \lspan\{ e_k : k \in \NN \}$ this equality is always true. Indeed, for each $x \in c_{00}$ there is some $n_x\in\NN$ fulfilling that
	\[
	x = \sum_{k=1}^{n_x} \ep{e_k^*}{x} e_k = \sum_{k=1}^{n_x} x_k e_k = \sum_{k\in\NN} x_k e_k.
	\]\\[-8pt]
\end{remark}

From now on we fix a natural number $N \in \NN$. We are going to construct:
\begin{enumerate}[--]
	\item a {\em projection} $P$ from $X$ into itself;
	
	\item a {\em sequence of functionals} $(w_k^*)_{k\geq N+2} \subset X^*$;
	
	\item two {\em sequences} $(m_k)_{k\in\NN} \in \NN^{\NN}$ and $(\lambda_k)_{k\in\NN} \in \TT^{\NN}$;
	
	\item and an {\em operator} $R$ on $X$.
\end{enumerate}
These auxiliary objects will then be used in \eqref{eq:def.T} to define the final operator $T$ so that $T_{(N)}$ is recurrent while $T_{(N+1)}$ is not. This will finish the proof of Theorem~\ref{The:N-counter-example}.

\subsection{The projection $P$ and the sequences $(w_k^*)_{k\geq N+2}$, $(m_k)_{k\in\NN}$ and $(\lambda_k)_{k\in\NN}$}


Denote by \ $P:X\longrightarrow \lspan\{e_1,e_2,...,e_{N+1}\}$ \ the projection of $X$ onto the linear span of the first $N+1$ basis vectors defined by
\[
Px := \sum_{j=1}^{N+1} \ep{e_j^*}{x} e_j \quad \text{ for every } x \in X.
\]
Note that $P$ is continuous. In fact $\|P\| \leq (N+1)K$. 
Let $E := \lspan\{e_1^*,e_2^*,...,e_{N+1}^*\}$ endowed with the norm $\|\cdot\|^*$ of $X^*$, and consider on $E$ the equivalent norm $\|\cdot\|_{\infty}$ defined in the following way for each $\alpha_1,\alpha_2,...,\alpha_{N+1} \in \CC$,
\[
\left\| \sum_{j=1}^{N+1} \alpha_j e_j^* \right\|_{\infty} := \max\{|\alpha_1|,|\alpha_2|,...,|\alpha_{N+1}|\}.
\]
This makes sense because the vectors $e_1^*,...,e_{N+1}^*$ are linearly independent. There exist constants $M \geq m >0$ such that
\[
m \|w^*\|_{\infty} \leq \|w^*\|^* \leq M \|w^*\|_{\infty} \quad \text{ for every } w^* \in E.
\]
Moreover, since the unit sphere $S_{E_{\infty}} := \{ w^* \in E : \|w^*\|_{\infty} = 1 \}$ is a compact metrizable space there exists a sequence $(w_k^*)_{k\geq N+2} \subset S_{E_{\infty}}$, which is dense in $S_{E_{\infty}}$. We include here a fact which will be used in the sequel:

\begin{fact}\label{Fact:find.i}
	For each $w^* \in S_{E_{\infty}}$ there exists an index $i \in \{1,2,...,N+1\}$ such that the following holds: for every $\eps>0$ and every $x \in X$ with $\|x-e_i\|<\eps$ we have that
	\[
	\left|\ep{w^*}{Px}\right| > 1 - (N+1)K\eps.
	\]
\end{fact}
\begin{proof}
	Since $w^* \in S_{E_{\infty}}$ there are coefficients $(\alpha_j)_{j=1}^{N+1} \in \CC^{N+1}$ such that $w^* = \sum_{j=1}^{N+1} \alpha_j e_j^*$, and there exists at least one index $i \in \{1,2,...,N+1\}$ with $|\alpha_i|=1$. Given $\eps>0$ and $x \in X$ with $\| x - e_i \| < \eps$, the inequality \eqref{eq:2K.eps} implies that $|x_i-1|<K\eps$ while $|x_j|<K\eps$ for every $1\leq j\leq N+1$ with $j\neq i$. Since $0 \leq |\alpha_j| \leq 1$ for every $1\leq j\leq N+1$ we get that
	\[
	\left|\ep{w^*}{Px}\right| = \left| \sum_{j=1}^{N+1} \alpha_j x_j \right| \geq \left|x_i\right| - \sum_{j=1,j\neq i}^{N+1} |\alpha_j| \left|x_j\right| > (1-K\eps) - NK\eps = 1 - (N+1)K\eps.\qedhere
	\]
\end{proof}


Let $(m_k)_{k\in\NN} \in \NN^{\NN}$ be a sequence of positive integers with the following properties:
\begin{enumerate}[(a)]
	\item $m_k \mid m_{k+1}$ for each $k \geq 1$;
	
	\item $m_1 = m_2 = \cdots = m_{N+1} = 1$;
\end{enumerate}
and starting from $k=N+2$, the sequence $(m_k)_{k\geq N+2}$ grows fast enough to satisfy:
\begin{enumerate}[(a)]
	\setcounter{enumi}{2}
	\item $\displaystyle \lim_{j\to\infty} \textstyle \left(m_j \cdot \sum_{k=j+1}^{\infty} \frac{1}{m_k} \right) = 0$.
\end{enumerate}
The sequence $(\lambda_k)_{k\in\NN} \in \TT^{\NN}$ is then defined by setting $\lambda_k := e^{2\pi i \frac{1}{m_k}}$ for each $k \in \NN$. Using the inequality $|e^{i\theta}-1| \leq |\theta|$, which is true for every $\theta \in \RR$, 
and also the fact that $\lambda_1=\lambda_2=\cdots=\lambda_{N+1}=1$, we deduce the inequality
\begin{equation}\label{eq:sum.lambda_k-1}
	\sum_{k=1}^{\infty} |\lambda_k-1| \leq 2\pi \sum_{k=N+2}^{\infty} \frac{1}{m_k} < \infty,
\end{equation}
since condition (c) on the sequence $(m_k)_{k\in\NN}$ implies that the series $\sum_{k=1}^{\infty} \frac{1}{m_k}$ is convergent.

\subsection{The operators $R$ and $T$}

For each $x \in c_{00} = \lspan\{ e_k : k \in \NN \}$, which we write as $x = \sum_{k=1}^{n_x} x_k e_k$, we set
\[
Rx := \sum_{k=1}^{n_x} \lambda_kx_ke_k.
\]
This defines a bounded operator $R$ on $c_{00}$. Indeed, for every $x = \sum_{k=1}^{n_x} x_ke_k \in c_{00}$,
\[
\|Rx\| \leq \|Rx-x\| + \|x\| \leq \sum_{k=1}^{n_x} |\lambda_k-1| \cdot \|x_ke_k\| + \|x\| \leq \left( K \sum_{k=1}^{\infty} |\lambda_k-1| + 1 \right) \|x\|
\]
by \eqref{eq:1K.norma.x}. Using the inequality \eqref{eq:sum.lambda_k-1} we see that the map $R$ extends to a bounded operator on $X$ still denoted by $R$. We now define the operator $T$ on $X$ by setting
\begin{equation}\label{eq:def.T}
	Tx := Rx + \sum_{k=N+2}^{\infty} \frac{1}{m_{k-1}} \ep{w_k^*}{Px} e_k \quad \text{ for every } x \in X.
\end{equation}
The second sum in the expression \eqref{eq:def.T} defines a bounded operator by the assumption (c) on the sequence $(m_k)_{k\in\NN}$, the fact that $\|P\| \leq (N+1)K$ and also that $\|w_k^*\|^* \leq M$ for each $k\geq N+2$. Indeed,
\[
\left\| \sum_{k=N+2}^{\infty} \frac{1}{m_{k-1}} \ep{w_k^*}{Px} e_k \right\| \leq \sum_{k=N+2}^{\infty} \frac{\|w_k^*\|^* \cdot \|Px\|}{m_{k-1}} \leq \left(M \cdot (N+1) K \cdot \sum_{k=N+2}^{\infty} \frac{1}{m_{k-1}}\right) \|x\|,
\]
for every $x \in X$, where the last parenthesis has finite value by assumption (c) on $(m_k)_{k \in \NN}$. It follows that $T$ is a continuous operator on $X$. Let us now compute its $n$-th power:

\begin{fact}[\textbf{Modification of \cite[Lemma 3.5]{Auge2012}}]\label{Fact:T.power.n}
	For every $x \in X$ and $n\geq 1$ we have that
	\[
	T^nx = R^nx + \sum_{k=N+2}^{\infty} \frac{\lambda_{k,n}}{m_{k-1}} \ep{w_k^*}{Px} e_k,
	\]
	where $\lambda_{k,n} := \sum_{l=0}^{n-1} \lambda_k^l = \frac{\lambda_k^n-1}{\lambda_k-1}$ for each $k \geq N+2$.
\end{fact}
\begin{proof}
	Suppose that the formula holds for some $n\geq 1$. Then
	\begin{align*}
		T^{n+1}x &= TR^nx + \sum_{k=N+2}^{\infty} \frac{\lambda_{k,n}}{m_{k-1}} \ep{w_k^*}{Px} Te_k = R^{n+1}x + \sum_{k=N+2}^{\infty} \frac{1}{m_{k-1}} \ep{w_k^*}{PR^nx} e_k \\[12pt]
		&+ \sum_{k=N+2}^{\infty} \frac{\lambda_{k,n}}{m_{k-1}} \ep{w_k^*}{Px} \left(Re_k + \sum_{j=N+2}^{\infty} \frac{1}{m_{j-1}} \ep{w_j^*}{Pe_k}e_j\right) \\[12pt]
		&= R^{n+1}x + \sum_{k=N+2}^{\infty} \frac{1 + \lambda_{k,n} \cdot \lambda_k}{m_{k-1}} \ep{w_k^*}{Px} e_k = R^{n+1}x + \sum_{k=N+2}^{\infty} \frac{\lambda_{k,n+1}}{m_{k-1}} \ep{w_k^*}{Px} e_k.
	\end{align*}
	In order to obtain these equalities we have used the fact that $PR^nx = Px$ (which follows for every $x \in c_{00}$ from the equalities $\lambda_1 = \lambda_2 = \cdots = \lambda_{N+1} = 1$, and then for every $x \in X$ by the density of $c_{00}$ in $X$), but also the fact that $Pe_k=0$ for all $k\geq N+2$.
\end{proof}

We will also need the following properties regarding the numbers $\lambda_{k,n}$:

\begin{fact}[\textbf{Modification of \cite[Fact 3.6]{Auge2012}}]\label{Fact:lambda_k_n}
	Let $n\geq 1$. Then:
	\begin{enumerate}[{\em(i)}]
		\item $|\lambda_{k,n}| \leq n$ for all $k \geq N+2$;
		
		\item $\lambda_{k,m_n}=0$ whenever $n\geq k\geq N+2$;
		
		\item $|\lambda_{k,n}|\geq \frac{2}{\pi}n > \frac{m_{k-1}}{\pi}$ whenever $k = \min\{ j \geq N+2 : 2n \leq m_j \}$.
	\end{enumerate}
\end{fact}
\begin{proof}
	Each $\lambda_{k,n}$, for $k\geq N+2$, is the geometric sum
	\[
	\textstyle \lambda_{k,n} = \sum_{l=0}^{n-1} \lambda_k^l = \frac{\lambda_k^n-1}{\lambda_k-1} \quad \text{ with } \lambda_k = e^{2\pi i \frac{1}{m_k}};
	\]
	then (i) follows from the triangle inequality; (ii) since $n\geq k$ implies that $\lambda_k^{m_n} = e^{2\pi i \frac{m_n}{m_k}} = 1$ by condition (a) on the sequence $(m_k)_{k\in\NN}$; and (iii) because
	\[
	|\lambda_{k,n}| = \left| \frac{e^{2\pi i \frac{n}{m_k}}-1}{e^{2\pi i \frac{1}{m_k}}-1} \right| = \frac{\left| \sin(\pi \frac{n}{m_k}) \right|}{\left| \sin(\pi \frac{1}{m_k}) \right|} \geq \frac{2}{\pi} n > \frac{m_{k-1}}{\pi},
	\]
	using that \ $\sin(\theta) \geq \frac{2}{\pi} \theta$ \ for each $\theta \in [0,\frac{\pi}{2}]$ and that \ $|\sin(\theta)| \leq |\theta|$ \ for every $\theta \in \RR$.
\end{proof}

Up to now, we have defined the operator $T$ and computed its powers $T^n$. It remains to check that $T_{(N)}$ is recurrent while $T_{(N+1)}$ is not. We start by showing that $R:X\longrightarrow X$ is  a rigid operator (see Subsection~\ref{SubSec:2.1qr} or \cite[Definition 1.2]{CoMaPa2014}) with respect to the (eventually) increasing sequence of positive integers $(m_k)_{k\in\NN}$:

\begin{fact}[\textbf{Modification of \cite[Lemma~3.4]{Auge2012}}]\label{Fact:R.rigid}
	We have $\displaystyle \lim_{k\to\infty} R^{m_k}x = x$ for all $x\in X$.
\end{fact}
\begin{proof}
	Following \cite[Lemma 3.4]{Auge2012} we first show that the sequence $(\|R^{m_j}\|)_{j\in\NN}$ is bounded. In fact, if we fix $x=\sum_{k=1}^n x_ke_k \in c_{00}$ such that $\|x\|=1$, for all $j\geq 1$ we have that
	\begin{eqnarray*}
		\|R^{m_j}x-x\| &=& \left\| \sum_{k=1}^n (\lambda_k^{m_j}-1)x_ke_k \right\| \nonumber\\[10pt]
		&=& \left\| \sum_{k=j+1}^n (\lambda_k^{m_j}-1)x_ke_k \right\| \qquad \text{since } \lambda^{m_j}=1 \text{ for all } k\leq j, \nonumber\\[10pt]
		&\leq& K \|x\| \sum_{k=j+1}^n |\lambda_k^{m_j}-1| = K \sum_{k=j+1}^n \left|e^{2\pi i \frac{m_j}{m_k}}-1\right| \quad \text{ by \eqref{eq:1K.norma.x},}\nonumber\\[10pt]
		&\leq& 2\pi K \sum_{k=j+1}^{\infty} \frac{m_j}{m_k} \qquad \text{since } |e^{i\theta}-1| \leq |\theta| \text{ for every } \theta \in \RR, \nonumber
	\end{eqnarray*}
	which is less than a constant independent of $j$ by condition (c) on $(m_k)_{k\in\NN}$. The density of $c_{00}$ in $X$ implies that $(\|R^{m_j}\|)_{j\in\NN}$ is a bounded sequence. Now if given $x \in X$ and $\eps>0$ we find $y \in c_{00}$ with $\|y-x\|<\eps$, then
	\[
	\|R^{m_j}x-x\| \leq \|R^{m_j}(x-y)\| + \|R^{m_j}y-y\| + \|y-x\| < \sup_{j\in\NN}\|R^{m_j}\|\eps + \|R^{m_j}y-y\| + \eps.
	\]
	The claim follows since $R^{m_j}y = y$ for every $j \in \NN$ large enough, and $\eps$ was arbitrary.
\end{proof}

Now we can prove the {\em recurrence} property of $T_{(N)}$:

\begin{proposition}
	The operator $T_{(N)}:X^N \longrightarrow X^N$ is recurrent since $\Rec(T_{(N)})=X^N$.
\end{proposition}
\begin{proof}
	Fix $x^{(1)},x^{(2)},...,x^{(N)} \in X$ and set $x_j^{(i)} := \ep{e_j^*}{x^{(i)}}$ for each $1\leq i\leq N$ and $j\geq 1$. Choose $(\alpha_j)_{j=1}^{N+1} \in \CC^{N+1}$ such that
	\[
	\begin{pmatrix} x_1^{(1)} & x_2^{(1)} & \cdots & x_{N+1}^{(1)} \\ x_1^{(2)} & x_2^{(2)} & \cdots & x_{N+1}^{(2)} \\ \vdots & \vdots & \cdots & \vdots \\  x_1^{(N)} & x_2^{(N)} & \cdots & x_{N+1}^{(N)} \end{pmatrix} \begin{pmatrix} \alpha_1 \\ \alpha_2 \\ \vdots \\ \alpha_{N+1} \end{pmatrix} = \begin{pmatrix} 0 \\ 0 \\ \vdots \\ 0 \end{pmatrix} \quad \text{ and } \quad \max_{1\leq j\leq N+1} |\alpha_j| = 1.
	\]
	This is indeed possible since the system considered has $N$ equations and $N+1$ unknowns. Set $w^*:=\sum_{j=1}^{N+1} \alpha_je_j^* \in E$. Then $w^* \in S_{E_{\infty}}$ and clearly $\ep{w^*}{Px^{(i)}}=0$ for all $1\leq i\leq N$. Let $(k_n)_{n\in\NN}$ be an increasing sequence of positive integers (all greater than $N+2$) such that $\lim_{n\to\infty} w_{k_n}^* = w^*$ in $E$. By Fact~\ref{Fact:T.power.n}, for each $1 \leq i \leq N$ we have that
	\begin{align*}
		\left(T^{m_{k_n-1}}-R^{m_{k_n-1}}\right)x^{(i)} &= \sum_{k=N+2}^{\infty} \frac{\lambda_{k,m_{k_n-1}}}{m_{k-1}} \ep{w_k^*}{Px^{(i)}} e_k = \underbrace{\sum_{N+2\leq k < k_n} \frac{\lambda_{k,m_{k_n-1}}}{m_{k-1}} \ep{w_k^*}{Px^{(i)}} e_k}_{(*)} \\[8pt]
		&+  \underbrace{\frac{\lambda_{k_n,m_{k_n-1}}}{m_{k_n-1}} \ep{w_{k_n}^*}{Px^{(i)}} e_{k_n}}_{(**)} + \underbrace{\sum_{k_n < k} \frac{\lambda_{k,m_{k_n-1}}}{m_{k-1}} \ep{w_k^*}{Px^{(i)}} e_k}_{(***)}.
	\end{align*}
	Note that for each $1\leq i\leq N$,
	\[
	(*) = \sum_{N+2\leq k < k_n} \frac{\lambda_{k,m_{k_n-1}}}{m_{k-1}} \ep{w_k^*}{Px^{(i)}} e_k = 0,
	\]
	since by (ii) of Fact~\ref{Fact:lambda_k_n} we have that $\lambda_{k,m_{k_n-1}} = 0$ for $k_n-1 \geq k$. By (i) of Fact~\ref{Fact:lambda_k_n} we have that $|\lambda_{k_n,m_{k_n-1}}| \leq m_{k_n-1}$ and hence
	\[
	\|(**)\| = \left| \frac{\lambda_{k_n,m_{k_n-1}}}{m_{k_n-1}} \ep{w_{k_n}^*}{Px^{(i)}} \right| \leq |\ep{w_{k_n}^*}{Px^{(i)}}| \underset{n\to\infty}{\longrightarrow} 0,
	\]
	since $\lim_{n\to\infty} w_{k_n}^* = w^*$ and $\ep{w^*}{Px^{(i)}}=0$ for every $1\leq i\leq N$. Moreover, by condition (c) on the sequence $(m_k)_{k\in\NN}$ we get that, for each $1\leq i\leq N$,
	\[
	\|(***)\| \leq \sum_{k_n < k} \frac{\left|\lambda_{k,m_{k_n-1}}\right|}{m_{k-1}} \|w_k^*\|^* \cdot \|P\| \cdot \|x^{(i)}\| \leq M \cdot (N+1)K \cdot \|x^{(i)}\| \cdot \sum_{k_n < k} \frac{m_{k_n-1}}{m_{k-1}} \underset{n\to \infty}{\longrightarrow} 0.
	\]
	Finally, for each $1\leq i \leq N$ we have that
	\begin{eqnarray*}
		\|T^{m_{k_n-1}}x^{(i)} - x^{(i)}\| &\leq& \|\left(T^{m_{k_n-1}}-R^{m_{k_n-1}}\right)x^{(i)}\| + \|R^{m_{k_n-1}}x^{(i)} - x^{(i)}\| \nonumber\\[8pt]
		&\leq& \|(*)\| + \|(**)\| + \|(***)\| + \|R^{m_{k_n-1}}x^{(i)} - x^{(i)}\| \underset{n\to\infty}{\longrightarrow} 0, \nonumber
	\end{eqnarray*}
	where $\|R^{m_{k_n-1}}x^{(i)}-x^{(i)}\| \to 0$ by Fact~\ref{Fact:R.rigid}. That is, $(x^{(1)},x^{(2)},...,x^{(N)}) \in \Rec(T_{(N)})$ and the arbitrariness of the vectors $x^{(1)}$, $x^{(2)}$, ..., $x^{(N)}$ implies that $\Rec(T_{(N)})=X^N$.
\end{proof}

To end the proof of Theorem~\ref{The:N-counter-example} we just have to show the following:

\begin{proposition}
	The operator $T_{(N+1)}:X^{N+1} \longrightarrow X^{N+1}$ is not recurrent.
\end{proposition}
\begin{proof}
	By contradiction suppose that $T_{(N+1)}$ is recurrent. Fix any positive value $\eps>0$ and consider the open balls $B(e_1,\eps),B(e_2,\eps),...,B(e_{N+1},\eps) \subset X$ centred at each vector $e_i$, $1\leq i\leq N+1$, with radius $\eps$. Then there exists some $n \in \NN$ such that
	\[
	T^n(B(e_i,\eps)) \cap B(e_i,\eps) \neq \varnothing \quad \text{ for every } i=1,2,...,N+1.
	\]
	Let $k_n := \min\{ j \geq N+2 : 2n \leq m_j \} \geq N+2$. By (iii) of Fact~\ref{Fact:lambda_k_n} we have that
	\begin{equation}\label{eq:lambda_k_n}
		|\lambda_{k_n,n}| > \tfrac{m_{k_n-1}}{\pi}.
	\end{equation}
	Moreover, Fact~\ref{Fact:find.i} applied to the functional $w_{k_n}^*$ yields the existence of (at least) one index $i_n \in \{1,2,...,N+1\}$ for which
	\begin{equation}\label{eq:functional}
		\left|\ep{w_{k_n}^*}{Px}\right| > 1-(N+1)K\eps \quad \text{ whenever } \|x-e_{i_n}\|<\eps.
	\end{equation}
	Picking any $x \in B(e_{i_n},\eps) \cap T^{-n}(B(e_{i_n},\eps)) \cap c_{00}$ we get the following inequalities
	\begin{eqnarray*}
		\eps &>& \|T^nx - e_{i_n}\| \geq \frac{1}{K} \left| \ep{e_{k_n}^*}{T^nx-e_{i_n}} \right| = \frac{1}{K} \left| \ep{e_{k_n}^*}{T^nx} \right| \qquad \text{ by \eqref{eq:1K.norma.x},} \nonumber\\[12pt]
		&=& \frac{1}{K} \left| \lambda_{k_n}^n x_{k_n} + \frac{\lambda_{k_n,n}}{m_{k_n-1}} \ep{w_{k_n}^*}{Px} \right| \qquad\qquad\qquad\qquad\quad \text{ by Fact~\ref{Fact:T.power.n}.} \nonumber
	\end{eqnarray*}
	Hence
	\begin{eqnarray*}
		\eps &>& \frac{1}{K} \left( \left| \frac{\lambda_{k_n,n}}{m_{k_n-1}} \right| \cdot \left|\ep{w_{k_n}^*}{Px}\right| - K\eps \right) > \frac{1}{K\pi} \cdot |\ep{w_{k_n}^*}{Px}| - \eps \qquad \text{ by \eqref{eq:2K.eps} and \eqref{eq:lambda_k_n},} \nonumber\\[12pt]
		&>& \frac{\left(1-(N+1)K\eps\right)}{K\pi} - \eps = \frac{1}{K\pi} - \left(\frac{N+1+\pi}{\pi}\right) \eps \qquad\qquad\qquad\qquad\qquad \text{ by \eqref{eq:functional}}.\nonumber
	\end{eqnarray*}
	The arbitrariness of $\eps>0$ yields now a contradiction by letting $\eps\to 0$.  
\end{proof}

The \textbf{complex} version of Theorem~\ref{The:N-counter-example} is now proved. The construction can be adapted to the \textbf{real} case by using the same arguments as in \cite[Section~3.2]{Auge2012}, so that every separable infinite-dimensional Banach space supports a recurrent operator which is not quasi-rigid. Moreover, there are operators whose $N$-fold direct sum is recurrent while its $N+1$-fold direct sum is not, and taking $N=1$ we solve negatively the {\em $T\oplus T$-recurrence problem}. It is worth mentioning that a slight variation of this construction can be used to solve the respective problem for {\em C$_0$-semigroups of operators}, which has been recently stated in \cite{ChenKosVe2021}. This was communicated to us by F. Bayart and Q. Menet independently, and the details of this variation will appear in a forthcoming work.

\section{Furstenberg families for pointwise recurrence}\label{Sec:4F-rec}

In the last two decades hypercyclicity has been studied from the {\em frequency of visits} point of view: one investigates ``how often'' the orbit of a vector returns to every open subset of the space. The recent notions of {\em frequent}, {\em $\Uc$-frequent} or {\em reiterative hypercyclicity} have emerged following this line of thought (see \cite{BaGri2006,Shkarin2009,BMPP2016,BMPP2019}). As shown in \cite{BMPP2016,BoGrLoPe2022}, this kind of ``{\em frequently dense behaviour}'' can be decomposed in two necessary ingredients:
\begin{enumerate}[1)]
	\item {\em usual hypercyclicity} (we have to require that the orbit studied is dense);
	
	\item some kind of {\em strong recurrence} (we have to require that the orbit studied returns to each neighbourhood of the starting vector with some {\em frequency}).
\end{enumerate}
Even though the study of {\em linear recurrence} began in 2014 with \cite{CoMaPa2014}, {\em stronger recurrence} notions were only systematically considered in the 2022 work \cite{BoGrLoPe2022}, and then in the following papers \cite{GriLo2023,CarMur2022_arXiv} where the concepts of {\em $\Fc$-recurrence} and {\em $\Fc$-hypercyclicity} with respect to a Furstenberg family $\Fc$ are deeply studied. In this section we show that quasi-rigidity can be studied from an $\Fc$-recurrence perspective for a very particular kind of Furstenberg families called {\em free filters} (see Proposition~\ref{Pro:qr.filter}). Then we study both the weakest and strongest possible $\Fc$-recurrence notions together with the Furstenberg families associated to them (see Subsection~\ref{SubSec:4.2apropriate.F}), and we finish the section recalling that $\Fc$-recurrence behaves well under homomorphisms, quasi-conjugacies and commutants (see Subsection~\ref{SubSec:4.3qc.c}).\\[-5pt]

Let us start by recalling the precise definitions, as well as some known examples.

\subsection{Definitions, examples and $\Fc(A)$-recurrence}

A collection of sets $\Fc \subset \Part(\NN_0)$ is said to be a {\em Furstenberg family} (a {\em family} for short) provided that: each set $A \in \Fc$ is infinite; $\Fc$ is hereditarily upward (i.e.\ $B \in \Fc$ whenever $A \in \Fc$ and $A \subset B$); and also $A\cap[n,\infty[ \in \Fc$ for all $A \in \Fc$ and $n \in \NN$. The {\em dual family} $\Fc^*$ for a given family $\Fc$ is the collection of sets $A \subset \NN_0$ such that $A \cap B\neq\varnothing$ for every $B \in \Fc$. The rather vague notion of {\em frequency} mentioned above will be defined in terms of families.\\[-5pt]

Given a dynamical system $(X,T)$, a point $x \in X$ and a non-empty (open) subset $U \subset X$ we denote by
\[
N_T(x,U) := \{ n \in \NN_0 : T^nx \in U \},
\]
the {\em return set from $x$ to $U$}. It will be denoted by $N(x,U)$ if no confusion with the map arises. Following \cite{BMPP2016,BMPP2019,BG2018,BoGrLoPe2022,CarMur2022_IEOT,CarMur2022_MS,CarMur2022_arXiv,GriLo2023,Lopez2024_RinM,Lopez2025_arXiv_frequently} we can now define:

\begin{definition}\label{Def:F-rec-hyp}
	Let $(X,T)$ be a dynamical system and let $\Fc$ be a Furstenberg family. A point $x \in X$ is said to be:
	\begin{enumerate}[--]
		\item {\em $\Fc$-recurrent} if $N(x,U) \in \Fc$ for every neighbourhood $U$ of $x$;
		
		\item {\em $\Fc$-hypercyclic} if $N(x,U) \in \Fc$ for every non-empty open subset $U \subset X$.
	\end{enumerate}
	Moreover, we will denote by:
	\begin{enumerate}[--]
		\item $\Fc\Rec(T)$ the {\em set of $\Fc$-recurrent points}, and we will say that $T$ is {\em $\Fc$-recurrent} whenever the set $\Fc\Rec(T)$ is dense in $X$;
		
		\item $\Fc\HC(T)$ the {\em set of $\Fc$-hypercyclic points}, and we will say that $T$ is {\em $\Fc$-hypercyclic} whenever the set $\Fc\HC(T)$ is non-empty.
	\end{enumerate}
\end{definition}

\begin{example}\label{Exa:F-rec-hyp}
	The Linear Dynamics works \cite{BaMa2009_book,BMPP2016,BMPP2019,BG2018,BoGrLoPe2022,CarMur2022_IEOT,CarMur2022_MS,CarMur2022_arXiv,CoMaPa2014,GriLo2023,GrPe2011_book,Lopez2024_RinM} have considered the Furstenberg families formed by:
	\begin{enumerate}[(a)]
		\item the {\em infinite sets}
		\[
		\Ic = \{ A \subset \NN_0 : A \text{ is infinite} \}.
		\]
		The notion of $\Ic$-recurrence (resp.\ $\Ic$-hypercyclicity) coincides with usual recurrence (resp.\ usual hypercyclicity) as defined in the Introduction and, as stated there, we will denote by $\Rec(T)$ (resp.\ $\HC(T)$) the set of recurrent (resp.\ hypercyclic) vectors.
		
		\item the {\em sets containing arbitrarily long arithmetic progressions}
		\[
		\AP := \{ A \subset \NN_0 : A \text{ contains arbitrarily long arithmetic progressions} \}.
		\]
		In other words, $A \in \AP$ if $\forall l \in \NN, \ \exists a,n \in \NN$ such that $\{a+kn : 0\leq k\leq l\} \subset A$, or equivalently $A$ contains the {\em arithmetic progression} of length $l+1$, common difference $n\in\NN$ and initial term $a\in\NN_0$. The concept of $\AP$-recurrence is equivalent to the notion of (topological) {\em multiple recurrence}, which requires that for each non-empty open subset $U \subset X$ and each length $l \in \NN$ there exists some $n\in\NN$ fulfilling that
		\[
		U \cap T^{-n}(U) \cap \cdots \cap T^{-ln}(U) \neq \varnothing.
		\]
		This concept was introduced for linear dynamical systems in \cite{CoPa2012} and further developed in \cite{KwiLiOpYe2017} and \cite{CarMur2022_MS}, where $\AP$-hypercyclicity is also considered. We will denote by $\AP\Rec(T)$ (resp.\ $\AP\HC(T)$) the set of $\AP$-recurrent (resp.\ $\AP$-hypercyclic) vectors.
		
		\item the {\em sets with positive upper Banach density} $\BDsup := \{ A \subset \NN_0 : \Bdsup(A)>0 \}$, where for each $A \subset \NN_0$ its {\em upper Banach density} is defined as
		\[
		\Bdsup(A) := \lim_{N \to \infty} \left( \max_{n \geq 0} \dfrac{\# \ A\cap [n+1,n+N]}{N} \right). \quad \text{See \cite{GTT2010} for alternative definitions.}
		\]
		The notion of $\BDsup$-recurrence (resp.\ $\BDsup$-hypercyclicity) has been called {\em reiterative recurrence} (resp.\ {\em reiterative hypercyclicity}), and $\RRec(T)$ (resp.\ $\RHC(T)$) is the set of reiteratively recurrent (resp.\ reiteratively hypercyclic) vectors, see \cite{BMPP2016,BMPP2019,BoGrLoPe2022,GriLo2023,Lopez2024_RinM}.
		
		\item the {\em sets with positive upper density} $\Dsup := \{ A \subset \NN_0 : \dsup(A)>0 \}$, where for each $A \subset \NN_0$ its {\em upper density} is defined as
		\[
		\dsup(A) := \limsup_{N \to \infty} \dfrac{\# \ A\cap [1,N]}{N}.
		\]
		The $\Dsup$-recurrence (resp.\ $\Dsup$-hypercyclicity) notion has been called {\em $\Uc$-frequent recurrence} (resp.\ {\em $\Uc$-frequent hypercyclicity}) and $\UFRec(T)$ (resp.\ $\UFHC(T)$) will denote the set of $\Uc$-frequent recurrent (resp.\ $\Uc$-frequent hypercyclic) vectors, see \cite{Shkarin2009,BMPP2016,BMPP2019,BoGrLoPe2022,GriLo2023}.
		
		\item the {\em sets with positive lower density} $\Dinf := \{ A \subset \NN_0 : \dinf(A)>0 \}$, where for each $A \subset \NN_0$ its {\em lower density} is defined as
		\[
		\dinf(A) := \limsup_{N \to \infty} \dfrac{\# \ A\cap [1,N]}{N}.
		\]
		The notion of $\Dinf$-recurrence (resp.\ $\Dinf$-hypercyclicity) is that of {\em frequent recurrence} (resp.\ {\em frequent hypercyclicity}), and $\FRec(T)$ (resp.\ $\FHC(T)$) is the set of frequently recurrent (resp.\ frequently hypercyclic) vectors, see \cite{BaGri2006,BMPP2016,BMPP2019,BoGrLoPe2022,GriLo2023,Lopez2024_RinM}.
		
		\item the {\em syndetic sets}
		\[
		\Sc := \{ A \subset \NN_0 : A \text{ is syndetic} \},
		\]
		i.e.\ $A \in \Sc$ whenever $\exists m \in \NN$ such that $\forall a \in \NN_0$, $[a,a+m]\cap A\neq\varnothing$. The notion of $\Sc$-recurrence has been called {\em uniform recurrence} and we denote by $\URec(T)$ the set of uniformly recurrent vectors, see \cite{Furstenberg1981_book,BoGrLoPe2022,GriLo2023,Lopez2024_RinM}. The family $\Sc$ coincides with the family of {\em sets with positive lower Banach density} $\BDinf := \{ A \subset \NN_0 : \Bdinf(A)>0 \}$, where for each $A \subset \NN_0$ its {\em lower Banach density} is defined as
		\[
		\Bdinf(A) := \lim_{N \to \infty} \left( \inf_{n \geq 0} \dfrac{\# \ A\cap [n+1,n+N]}{N} \right). \quad \text{See \cite{GTT2010} for alternative definitions.}
		\]
		
		\item the {\em IP-sets}
		\[
		\IP := \{ A \subset \NN_0 : A \text{ is an IP-set} \},
		\]
		i.e.\ $A \in \IP$ whenever $\{ \sum_{k \in F} n_k : F \subset \NN \text{ finite} \} \subset A$ for some sequence of natural numbers $(n_k)_{k\in\NN}$. In this case we use the dual family $\IP^*$, and for the $\IP^*$-recurrence notion we denote by $\IP^*\Rec(T)$ the set $\IP^*$-recurrent vectors, see \cite{Furstenberg1981_book,GaMaPeOp2015,BoGrLoPe2022,GriLo2023}.
		
		\item the {\em $\Delta$-sets}
		\[
		\Delta := \{ A \subset \NN_0 : A \text{ is a $\Delta$-set} \},
		\]
		i.e.\ $A \in \Delta$ whenever $\exists B \in \Ic$ with $(B-B)\cap\NN \subset A$. In this case we again use the dual family $\Delta^*$, and for the $\Delta^*$-recurrence notion we denote by $\Delta^*\Rec(T)$ the set of $\Delta^*$-recurrent vectors, see \cite{Furstenberg1981_book,GriLo2023}.
	\end{enumerate}
\end{example}

We have the following relations between the families listed above:
\[
\Del^* \subset \IP^* \subset \Sc = \BDinf \subset \Dinf \subset \Dsup \subset \BDsup \subset \AP \subset \Ic.
\]
Among these inclusions, some of them follow easily from the definitions while others depend on deep theorems like the celebrated {\em Szemer\'edi theorem}, see \cite{BerDown2008,CarMur2022_MS,HS1998,Szemeredi1975}. They imply the respective inclusions between the introduced sets of recurrent vectors
\begin{align*}
	\Del^*\Rec(T) &\subset \IP^*\Rec(T) \subset \URec(T) \subset ...\\[10pt]
	... &\subset \FRec(T) \subset \UFRec(T) \subset \RRec(T) \subset \AP\Rec(T) \subset \Rec(T),
\end{align*}
and also between the sets of hypercyclic vectors
\[
\FHC(T) \subset \UFHC(T) \subset \RHC(T) \subset \AP\HC(T) \subset \HC(T).
\]
It is worth mentioning that the families $\Fc$ for which there exist $\Fc$-hypercyclic operators are by far less common than those for which $\Fc$-recurrence exists, which is not surprising, since having an orbit distributed around the whole space is much more complicated than having it just coming back around the initial point of the orbit (it is known that there does not exist any $\Del^*$, $\IP^*$ neither $\Sc$-hypercyclic operator, see \cite{BMPP2016,BoGrLoPe2022}). Other families have been considered from the $\Fc$-hypercyclicity perspective in the works \cite{BG2018,ErEsMe2021}. Our aim now is to show that a dynamical system is {\em quasi-rigid} if and only if it is $\Fc$-recurrent with respect to a {\em free filter} $\Fc$. Let us recall the following classical definitions (see \cite{Bourbaki1989}):

\begin{definition}
	Let $\Fc \subset \Part(\NN_0)$ be a collection of non-empty sets of natural numbers (note that $\Fc$ is not necessarily a Furstenberg family here). We say that $\Fc$:
	\begin{enumerate}[--]
		\item has the {\em finite intersection property}, if $\bigcap_{A \in \Ac} A \neq \varnothing$ for every $\Ac \subset \Fc$ with $\#\Ac < \infty$;
		
		\item is a {\em filter}, if $\Fc$ is hereditarily upward and for every $A,B \in \Fc$ we have that $A \cap B \in \Fc$;
		
		\item is a {\em free filter}, if it is a filter and $\bigcap_{A \in \Fc} A = \varnothing$.
	\end{enumerate}
	Finally, given any infinite subset $A \subset \NN_0$ we denote by $\Fc(A)$ the {\em free filter generated by $A$}, which is the Furstenberg family
	\[
	\Fc(A) := \left\{ B \subset \NN_0 : \# (A\setminus B) < \infty \right\}.
	\]
	In the literature, $\Fc(A)$ has been called the {\em Fr\'echet filter} on $A$, or also the {\em eventuality filter} generated by the increasing sequence of integers $(n_k)_{k\in\NN}$ forming the set $A \subset \NN_0$ (see \cite{Bourbaki1989}).
\end{definition}

\begin{remark}
	It is clear that every {\em free filter} is a {\em filter}, and that every {\em filter} has the {\em finite intersection property}. However, when $\Fc$ is a {\em Furstenberg family} as defined in this paper the previous definitions have some extra (and immediate) \textbf{consequences}:
	\begin{enumerate}[(a)]
		\item {\em A Furstenberg family $\Fc$ has the finite intersection property if and only if $\bigcap_{A \in \Ac} A \in \Ic$ for every $\Ac \subset \Fc$ with $\#\Ac < \infty$}: indeed, if $\Fc$ has the finite intersection property but for some $\Ac \subset \Fc$ with $\#\Ac < \infty$ we had that $\bigcap_{A \in \Ac} A$ is finite, then for any fixed $n> \max\left(\bigcap_{A \in \Ac} A\right)$ we would arrive to the contradiction: 
		\[
		\Ac' := \{ A \cap [n,\infty[ : A \in \Ac \} \subset \Fc \quad \text{ with } \quad \#\Ac'<\infty \quad \text{ but } \quad \bigcap_{A \in \Ac'} A = \varnothing.
		\]
		
		\item {\em Let $\Fc$ be a Furstenberg family. Then, $\Fc$ is a filter if and only if it is a free filter}: note that given any $B$ from a filter $\Fc$ which is also a Furstenberg family we have that
		\[
		\bigcap_{A \in \Fc} A \subset \bigcap_{n \in \NN} B\cap[n,\infty[ = \varnothing,
		\]
		by the definition of Furstenberg family used in this paper.
	\end{enumerate}
\end{remark}

We are ready to characterize quasi-rigidity in terms of $\Fc$-recurrence and free filters. Recall first that, given a filter $\Fc \subset \Part(\NN_0)$, a collection of sets $\Ac \subset \Fc$ is called a {\em base~of~the~filter~$\Fc$} if for every set $B \in \Fc$ there exits some $A_B \in \Ac$ such that $A_B \subset B$. We have the following:

\begin{proposition}\label{Pro:qr.filter}
	Let $(X,T)$ be a dynamical system. The next statements are equivalent:
	\begin{enumerate}[{\em(i)}]
		\item $T$ is quasi-rigid;
		
		\item $T$ is $\Fc(A)$-recurrent for some infinite subset $A \subset \NN_0$;
		
		\item $T$ is $\Fc$-recurrent with respect to a free filter $\Fc$ with a countable base.
	\end{enumerate}
	Moreover, if $X$ is a second-countable space, the previous statements are equivalent to:
	\begin{enumerate}[{\em(iv)}]
		\item $T$ is $\Fc$-recurrent for a family $\Fc$ with the finite intersection property.
	\end{enumerate}
\end{proposition}
\begin{proof}
	For (i) $\Rightarrow$ (ii), if $T$ is quasi-rigid with respect to $(n_k)_{k\in\NN}$ it is enough to take $A:=\{ n_k : k \in\NN \}$. The implication (ii) $\Rightarrow$ (iii) is obvious since $\Fc(A)$ has a countable base. To see (iii) $\Rightarrow$ (i) assume that the map $T$ is $\Fc$-recurrent and let $(A_k)_{k\in\NN}$ be a decreasing base of the free filter $\Fc$. Setting $n_k := \min(A_k)$ for each $k \in \NN$, and taking a subsequence if necessary, we get an increasing sequence $(n_k)_{k\in\NN}$ with respect to which $T$ is quasi-rigid.\\[-5pt]
	
	We always have that (i), (ii) and (iii) $\Rightarrow$ (iv) even if the space $X$ is not second-countable. Assume now that $X$ is second-countable and that $T$ is $\Fc$-recurrent for a family $\Fc$ with the finite intersection property. Let $\{ x_s : s \in \NN \} \subset \Fc\Rec(T)$ be a dense set in $X$ and, for each $s \in \NN$, let $(U_{s,k})_{k\in\NN}$ be a decreasing neighbourhood basis of $x_s$. We can now recursively construct an increasing sequence of positive integers $(n_k)_{k\in\NN}$ such that
	\[
	T^{n_k}x_s \in U_{s,k} \quad \text{ for every } s,k \in \NN \text{ with } 1\leq s\leq k,
	\]
	by taking a sufficiently large integer $n_k \in \bigcap_{s=1}^k N_T(x_s,U_{s,k}) \in \Ic$. It is easily seen that $T$ is quasi-rigid with respect to $(n_k)_{k\in\NN}$, which shows that (iv) $\Rightarrow$ (i), (ii) and (iii).
\end{proof}

\subsection{Appropriate Furstenberg families for $\Fc$-recurrence}\label{SubSec:4.2apropriate.F}

We have just shown that quasi-rigidity is indeed a particular case of $\Fc$-recurrence for a Furstenberg family not listed in Example~\ref{Exa:F-rec-hyp}. It is then natural to ask the following:

\begin{question}\label{Q:families}
	Are the families of Example~\ref{Exa:F-rec-hyp}, i.e.\ $\Fc = \Ic, \AP, \BDsup, \Dsup, \Dinf, \Sc, \IP^*, \Del^*$, and the free filters $\Fc(A)$ the only ones for which $\Fc$-recurrence should be considered?
\end{question}

This may seem a very open query. Our objective in this part of the paper is to look in detail into two very different classes of Furstenberg families: those for which $\Fc$-recurrence is the {\em weakest} possible recurrence notion, i.e.\ families $\Fc \subsetneq \Ic$ such that $\Fc\Rec(T) = \Rec(T)$; and those for which $\Fc$-recurrence is the {\em strongest} possible recurrence notion, i.e.\ periodicity.\\[-5pt]

Starting with the first of these two classes: {\em for any dynamical system $(X,T)$ every recurrent point $x \in \Rec(T)$ is an $\IP$-recurrent point} (this is shown in \cite[Theorem~2.17]{Furstenberg1981_book} when $X$ is a {\em metric space}, but the proof easily extends for arbitrary dynamical systems). Hence $\Rec(T) \subset \IP\Rec(T) \subset \Del\Rec(T) \subset \Ic\Rec(T)$. Since $\Ic\Rec(T) = \Rec(T)$ we always have
\[
\IP\Rec(T) = \Del\Rec(T) = \Rec(T).
\]
This relation allows us to give an alternative (and much simpler) proof for the so-called {\em Ansari} and {\em Le\'on-M\"uller recurrence theorems}, which assert that powers and unimodular multiples of an operator share the same set of recurrent vectors (see \cite[Proposition~2.3]{CoMaPa2014} for the original, different for each of the $T^p$ and $\lambda T$ cases, and rather long proof):

\begin{proposition}[\textbf{\cite[Proposition~2.3]{CoMaPa2014}}]\label{Pro:Ansari+Leon-Muller}
	Let $(X,T)$ be a dynamical system. Then:
	\begin{enumerate}[{\em(a)}]
		\item For every positive integer $p \in \NN$ we have that $\Rec(T)=\Rec(T^p)$.
		
		\item If $(X,T)$ is linear, then for every $\lambda \in \TT$ we have that $\Rec(T)=\Rec(\lambda T)$.
	\end{enumerate}
\end{proposition}
\begin{proof}
	Let $x \in \Rec(T)$, i.e.\ $N_T(x,V)$ belongs to $\Delta$ for every neighbourhood $V$ of $x$. Fixed any $p \in \NN$ and any neighbourhood $U$ of $x$, since the set $(p \cdot \NN_0)$ belongs to $\Delta^*$ we have that
	\[
	\varnothing \neq \tfrac{1}{p} \cdot \left( N_T(x,U) \cap \left( p \cdot \NN_0 \right) \right) \subset N_{T^p}(x,U),
	\]
	which implies that $x \in \Rec(T^p)$, proving (a). To check (b) assume moreover that $(X,T)$ is a linear dynamical system and fix any $\lambda \in \TT$. Let $\eps>0$ and let $U_0 \subset U$ be a neighbourhood of $x$ such that $\mu \cdot U_0 \subset U$ for all $|\mu-1|<\eps$. Since the set $\{ n \in \NN : |\lambda^n-1|<\eps \}$ belongs to $\Delta^*$ (see for instance \cite[Proposition~4.1]{GriLo2023}) we have that
	\[
	\varnothing \neq N_T(x,U_0) \cap \{ n \in \NN : |\lambda^n-1|<\eps \} \subset N_{\lambda T}(x,U),
	\]
	which implies that $x \in \Rec(\lambda T)$.
\end{proof}

It is worth mentioning that, since the families $\Delta^*$ and $\IP^*$ are filters (see \cite{BerDown2008}), then the proof of Proposition~\ref{Pro:Ansari+Leon-Muller} can be adapted to show the equalities
\[
\Delta^*\Rec(T^p)=\Delta^*\Rec(T)=\Delta^*\Rec(\lambda T),
\]
and also
\[
\IP^*\Rec(T^p)=\IP^*\Rec(T)=\IP^*\Rec(\lambda T).
\]
The $\IP^*$-version of these inequalities was already discussed in \cite[Section~6]{BoGrLoPe2022}, in a different way via the so-called {\em product recurrent points}, but the $\Delta^*$-version is new as far as we know.\\[-5pt]

Let us now look at the {\em strongest} possible recurrence notion by using families formed by very big sets of natural numbers that have been also used in Linear Dynamics:

\begin{example}\label{Exa:extra.families}
	Consider the Furstenberg families formed by:
	\begin{enumerate}[(a)]
		\item the {\em cofinite sets} $\Ic^* := \{ A \subset \NN_0 : A \text{ is cofinite} \}$, which is the dual family of that formed by the infinite sets $\Ic$, and which is used to define the notion of (topological) {\em mixing}: a dynamical system $(X,T)$ is called {\em mixing} if the set $\{ n \in \NN_0 : T^{n}(U)\cap V\neq \varnothing \}$ belongs to $\Ic^*$ for every pair of non-empty open subset $U,V \subset X$ (see \cite[Definition~1.38]{GrPe2011_book}).
		
		\item the {\em thick sets} $\Tc := \{ A \subset \NN_0 : A \text{ is thick} \}$, i.e.\ $A \in \Tc$ if $\forall m \in \NN, \ \exists a_m \in A$ with $[a_m,a_m+m] \subset A$. This family is known to characterize the {\em weak-mixing} property (see \cite[Theorem~3]{GrPe2010} or \cite[Theorem~1.54]{GrPe2011_book}): {\em for every dynamical system $(X,T)$ the following statements are equivalent:
		\begin{enumerate}[{\em(i)}]
			\item $(X,T)$ is weakly mixing;
		
			\item the set $\{ n \in \NN_0 : T^{n}(U)\cap V\neq \varnothing \}$ belongs to $\Tc$ for every pair of non-empty open subsets $U,V \subset X$.
		\end{enumerate}}
		
		\item the {\em thickly syndetic sets} $\TS := \{ A \subset \NN_0 : A \text{ is thickly syndetic} \}$, i.e.\ $A \in \TS$ if $\forall m \in \NN, \ \exists A_m$ syndetic such that $A_m+[0,m]\subset A$. This family characterizes the {\em topological ergodicity} (i.e.\ the property that $\{ n \in \NN_0 : T^{n}(U)\cap V\neq \varnothing \}$ is syndetic for every pair of non-empty open subsets $U,V \subset X$) for the particular case of \textbf{linear} dynamical systems (see \cite[Exercise~2.5.4]{GrPe2011_book} and \cite{BMPP2019}): {\em for every linear dynamical system $(X,T)$ the following statements are equivalent:
			\begin{enumerate}[{\em(i)}]
				\item $(X,T)$ is topologically ergodic;
				
				\item the set $\{ n \in \NN_0 : T^{n}(U)\cap V\neq \varnothing \}$ belongs to $\TS$ for every pair of non-empty open subsets $U,V \subset X$.
		\end{enumerate}}
		
		\item the {\em sets of density greater than $\delta>0$}, which for each $0<\delta\leq 1$ are the families
		\begin{eqnarray}
			\BDinf_{\del} := \{ A \subset \NN_0 : \Bdinf(A)\geq\del \}, &\quad& \BDsup_{\del} := \{ A \subset \NN_0 : \Bdsup(A)\geq\del \},\nonumber\\[5pt]
			\Dinf_{\del} := \{ A \subset \NN_0 : \dinf(A)\geq\del \}, &\quad& \Dsup_{\del} := \{ A \subset \NN_0 : \dsup(A)\geq\del \}.\nonumber
		\end{eqnarray}
		The density families have also been studied in Linear Dynamics, and we refer the reader to \cite{BMPP2016,BMPP2019} for more about them.
	\end{enumerate}
\end{example}

We are about to show that the recurrence notions associated to the Furstenberg families introduced in Example~\ref{Exa:extra.families} above imply different periodicity notions (see Proposition~\ref{Pro:periodic-recurrence}). We replicate the arguments from \cite[Proposition~3]{BMPP2016}, which have been used in an independent way in \cite[Lemma~3.17 and Corollary~3.18]{CarMur2022_arXiv}. Our contribution here is to rewrite these results in their ``{\em pointwise $\Fc$-recurrence}'' version. We start by proving two key lemmas:

\begin{lemma}\label{Lem:periodic-points}
	Let $(X,T)$ be a dynamical system. Given $x \in X$ and $N \in \NN$ the following statements are equivalent:
	\begin{enumerate}[{\em(i)}]
		\item $x \in \Per(T)$ and its period is strictly lower than $N$;
		
		\item for every neighbourhood $U$ of $x$ there exists $n_U \in \NN_0$ such that
		\[
		\# \ N_T(x,U) \cap [n_U+1,n_U+N] \geq 2;
		\]
		
		\item for every neighbourhood $U$ of $x$ we have that $T^p(U)\cap U\neq\varnothing$ for some $1\leq p< N$.
	\end{enumerate}
\end{lemma}
\begin{proof}
	For (i) $\Rightarrow$ (ii) let $n_U$ be the period of $x$ minus one. For (ii) $\Rightarrow$ (iii) recall that
	\[
	\text{given } n_1<n_2 \in N_T(x,U) \quad \text{ we have that } \quad T^{n_2-n_1}(U)\cap U\neq\varnothing.
	\]
	Conversely, if we suppose that $T^px \neq x$ for $1\leq p<N$, by continuity we can find an open neighbourhood $U$ of $x$ such that $T^p(U)\cap U = \varnothing$ for every $1\leq p<N$, so (iii) $\Rightarrow$ (i).
\end{proof}

\begin{lemma}\label{Lem:BDsup-consecutive}
	Let $0<\del\leq 1$ and $N \in \NN$ with $\frac{1}{N}<\del$. Then, for every $A \in \BDsup_{\del}$ there is $n_A \in \NN_0$ such that
	\begin{equation*}
		\# \ A\cap [n_A+1,n_A+N]\geq 2.
	\end{equation*}
\end{lemma}
\begin{proof}
	Otherwise we would have that $\# \ A\cap [n+1,n+N] \leq 1$ for every $n \in \NN_0$ obtaining the contradiction
	\begin{equation*}
		\Bdsup(A) = \lim_{K \to \infty} \left( \max_{n \geq 0} \dfrac{\# \ A\cap [n+1,n+(N\cdot K)]}{N\cdot K} \right) \leq \lim_{K \to \infty} \dfrac{K}{N\cdot K} = \frac{1}{N} < \del.\qedhere
	\end{equation*}
\end{proof}

Finally we get the desired result, in which we classify the periodic points of a dynamical system in terms of the density of the return sets:

\begin{proposition}\label{Pro:periodic-recurrence}
	Let $(X,T)$ be a dynamical system and let $x \in X$. Then:
	\begin{enumerate}[{\em(a)}]
		\item Given any $0<\del\leq 1$, the following statements are equivalent:
		\begin{enumerate}[{\em(i)}]
			\item $x \in \Per(T)$ and its period is lower or equal to $\left\lfloor\frac{1}{\del}\right\rfloor$;
			
			\item $x \in \BDinf_{\del}\Rec(T)$;
			
			\item $x \in \Dinf_{\del}\Rec(T)$;
			
			\item $x \in \Dsup_{\del}\Rec(T)$;
			
			\item $x \in \BDsup_{\del}\Rec(T)$.
		\end{enumerate}
		In particular, if $T$ is $\BDsup_{\del}$-recurrent then $T^N = I$ for $N=1\cdot 2\cdots \left(\left\lfloor\frac{1}{\del}\right\rfloor-1\right)\cdot \left\lfloor\frac{1}{\del}\right\rfloor$.
		
		\item The following statements are equivalent:
		\begin{enumerate}[{\em(i)}]
			\item $x$ is a fixed point, i.e.\ $Tx=x$;
			
			\item $x \in \Ic^*\Rec(T)$;
			
			\item $x \in \TS\Rec(T)$;
			
			\item $x \in \Tc\Rec(T)$;
			
			\item $x \in \BDsup_{\del}\Rec(T)$ for some $\del>\frac{1}{2}$;
			
			\item for every neighbourhood $U$ of $x$ the set $N(x,U)$ contains two consecutive integers.
		\end{enumerate}
	\end{enumerate}
\end{proposition}
\begin{proof}
	(a): Let us first show that (i) $\Rightarrow$ (ii): let $p \in \left\{ 1,2,...,\left\lfloor\frac{1}{\del}\right\rfloor \right\}$ be such that $T^px=x$. Then 
	\[
	\Bdinf(N_T(x,U)) \geq \Bdinf(p\cdot\NN_0) = \frac{1}{p} \geq \delta
	\]
	for every neighbourhood $U$ of $x$, so $N(x,U) \in \BDsup_{\delta}$. For (ii) $\Rightarrow$ (iii) $\Rightarrow$ (iv) $\Rightarrow$ (v) recall that $\BDinf_{\del} \subset \Dinf_{\del} \subset \Dsup_{\del} \subset \BDsup_{\del}$. For (v) $\Rightarrow$ (i) let $N := \left\lfloor\frac{1}{\del}\right\rfloor + 1$. By Lemma~\ref{Lem:BDsup-consecutive}, for every neighbourhood $U$ of $x$ there is $n_U \in \NN$ such that
	\[
	\# \ N_T(x,U) \cap [n_U+1,n_U+N] \geq 2,
	\]
	so Lemma~\ref{Lem:periodic-points} finishes the work.\newpage
	
	(b): The implications (i) $\Rightarrow$ (ii) $\Rightarrow$ (iii) $\Rightarrow$ (iv) $\Rightarrow$ (v) follow from the well-known relations $\NN_0 \in \Ic^* \subset \TS \subset \Tc = \BDsup_1 \subset \BDsup_{\del}$ for every $0<\delta\leq 1$ (see \cite{BerDown2008,HS1998}). To prove (v) $\Rightarrow$ (vi) use Lemma~\ref{Lem:BDsup-consecutive} applied to $N=2$. Finally, (vi) $\Leftrightarrow$ (i) follows from Lemma~\ref{Lem:periodic-points}.
\end{proof}

The characterization obtained for periodic points in part (a) of Proposition~\ref{Pro:periodic-recurrence} tells us which are the Furstenberg families $\Fc$ whose respective $\Fc$-recurrence notion is trivial, in the sense that it coincides with periodicity. However, it is still interesting to study all the recurrence notions introduced in Example~\ref{Exa:F-rec-hyp} since clearly $\Per(T) \subset \Delta^*\Rec(T)$ for every dynamical system $(X,T)$. We have also reproved \cite[Proposition~3]{BMPP2016}:

\begin{corollary}
	There is no $\BDsup_{\delta}$-hypercyclic operator for any $0<\delta\leq 1$.
\end{corollary}

\subsection{Quasi-conjugacies and commutants in $\Fc$-recurrence}\label{SubSec:4.3qc.c}

We end Section~\ref{Sec:4F-rec} by presenting some basic but useful tools regarding {\em quasi-conjugacies} and (non-linear) {\em commutants} which will be used in Section~\ref{Sec:5lin+dense.lin} to study the {\em lineability} and {\em dense lineability} properties of the set of $\Fc$-recurrent vectors. Following \cite{Furstenberg1981_book,GrPe2011_book} we define:

\begin{definition}\label{Def:homeo.quasi-con.con}
	Given two dynamical systems $(X,T)$ and $(Y,S)$, we say that a map $\phi:Y\longrightarrow X$ is an {\em homomorphism of dynamical systems} if $\phi$ is continuous and the diagram 
	\begin{equation*}
		\begin{CD}
			Y  @>{S}>>  Y \\
			@V{\phi}VV  @VV{\phi}V \\
			X  @>{T}>>  X
		\end{CD}
	\end{equation*}
	commutes, i.e.\ $\phi\circ S = T\circ\phi$. Moreover, we say that the map $\phi$ is a:
	\begin{enumerate}[--]
		\item {\em quasi-conjugacy}, if $\phi$ has dense range, and hence that $(X,T)$ is {\em quasi-conjugate} to $(Y,S)$;
		
		\item {\em conjugacy}, if $\phi$ is an homeomorphism between $X$ and $Y$ (see Remark~\ref{Rem:complexification}).
	\end{enumerate}
\end{definition}

Many dynamical properties are preserved by quasi-conjugacy (see \cite[Chapter~1]{GrPe2011_book}): if $(Y,S)$ admits a dense orbit, is topologically transitive, weakly-mixing or even (Devaney) chaotic, then so is $(X,T)$. On the other hand, to preserve the {\em recurrence} of a point it is enough to have an homomorphism. This is a really well-known fact (see \cite[Proposition 1.3]{Furstenberg1981_book}), and indeed similar arguments were already used in \cite[Proposition 9.9]{Furstenberg1981_book} for $\Fc$-recurrence with respect to the Furstenberg families $\Del^*$ and $\IP^*$, or in \cite[Proposition~3.8]{CarMur2022_arXiv} for general Furstenberg families. We include here the general argument regarding return sets:

\begin{lemma}\label{Lem:homomorphism}
	Let $(X,T)$ and $(Y,S)$ be dynamical systems and let $\phi:Y\longrightarrow X$ be an homomorphism between them. Given $y\in Y$ and any neighbourhood $U \subset X$ of $\phi(y)$, the set $V:=\phi^{-1}(U) \subset Y$ is a neighbourhood of $y$ for which $N_S(y,V) = N_T(\phi(y),U)$. In particular, for any Furstenberg family $\Fc$ we have that:
	\begin{enumerate}[{\em(a)}]
		\item If $y \in Y$ is $\Fc$-recurrent for $S$ then $\phi(y)$ is $\Fc$-recurrent for $T$.
		
		\item If $\phi$ is a quasi-conjugacy and $S$ is $\Fc$-recurrent then so is $T$.
	\end{enumerate}
\end{lemma}
\begin{proof}
	By continuity of $\phi$ the set $V=\phi^{-1}(U)$ is a neighbourhood of $y$. Since $\phi\circ S = T\circ\phi$ we have that $\phi(S^ny) = T^n\phi(y)$ for every $n \in \NN_0$. It follows that
	\[
	N_T(\phi(y),U) = \{n\in\NN_0 : T^n\phi(y) \in U\} = \{n\in\NN_0 : S^ny \in \phi^{-1}(U)\} = N_S(y,V).
	\]
	In particular, if $y$ is $\Fc$-recurrent then the return set $N_T(\phi(y),U) = N_S(y,\phi^{-1}(U))$ belongs to $\Fc$ for every neighbourhood $U$ of $\phi(y)$. Statement (b) follows immediately from (a).
\end{proof}

Given a dynamical system $(X,T)$ and a continuous map $S:X\longrightarrow X$ commuting with $T$, then $S$ is homomorphism between $T$ and itself, and Lemma~\ref{Lem:homomorphism} implies the following:

\begin{corollary}\label{Cor:return-sets-commute}
	Let $(X,T)$ be a dynamical system and let $S:X\longrightarrow X$ be a continuous map commuting with $T$, that is $S \circ T = T \circ S$. Given $x \in X$ and any neighbourhood $U$ of $Sx$, the set $V:=S^{-1}(U)$ is a neighbourhood of $x$ for which $N_T(x,V) = N_T(Sx,U)$. In particular, for any Furstenberg family $\Fc$ we have that $S(\Fc\Rec(T)) \subset \Fc\Rec(T)$.
\end{corollary}

Corollary~\ref{Cor:return-sets-commute} unifies some of the well-known arguments already used in \cite[Proposition~2.5]{AmBe2019_arXiv}, \cite[Proof of Theorem 9.1]{CoMaPa2014} and \cite[Propositions 1.3 and 9.9]{Furstenberg1981_book}, generalizing them to arbitrary Furstenberg families and (not necessarily linear) maps acting on arbitrary topological spaces. From now on we will write
\[
\Ci(X) := \left\{ S:X\longrightarrow X \text{ continuous map} \right\},
\]
i.e.\ the {\em set of continuous} (and not necessarily linear) {\em self-maps} of the Hausdorff topological space $X$, and following the notation of \cite{Jungck1976} we define:

\begin{definition}
	Let $(X,T)$ be a dynamical system. The (non-linear) {\em commutant} of $T$ is the set
	\[
	\Com_T := \left\{ S \in \Ci(X) : S \circ T = T \circ S \right\}.
	\]
	Given $x \in X$ we define the {\em $\Com_T$-orbit} of $x$ as $\Com_T(x) := \{ Sx : S \in \Com_T \}$. We say that a subset $Y \subset X$ is {\em $\Com_T$-invariant} if $S(Y) \subset Y$ for every $S \in \Com_T$.
\end{definition}

\begin{remark}
	Since $T \in \Com_T$ we always have that:
	\begin{enumerate}[--]
		\item every $\Com_T$-invariant set is also $T$-invariant;
		
		\item if $X$ is a linear space, then $\{ p(T) : p \text{ polynomial} \} \subset \Com_T$ and for every $x \in X$ we get that
		\[
		\lspan(\orb(x,T)) = \lspan\{T^nx : n \in \NN_0\} = \{ p(T)x : p \text{ polynomial} \} \subset \Com_T(x),
		\]
		so that every {\em cyclic} vector for an operator $T \in \Lc(X)$ has a {\em dense $\Com_T$-orbit}. Moreover:
	\end{enumerate}
\end{remark}

\begin{lemma}\label{Lem:Com_T-subspace}
	Let $(X,T)$ be a linear dynamical system. Then $(\Com_T,+,\circ)$ is a subring of the ring of continuous maps $(\Ci(X),+,\circ)$. In particular, given any $x \in X$ the set $\Com_T(x)$ is a vector subspace of $X$ and the smallest $\Com_T$-invariant subset of $X$ containing the point $x$.
\end{lemma}
\begin{proof}
	Given $S,R \in \Com_T$ and $\alpha,\beta \in \KK$ we have that $(\alpha S+\beta R) \in \Com_T$ and $(S \circ R) \in \Com_T$. In particular, given any $x \in X$ the set $\Com_T(x)$ is a vector subspace of $X$. Moreover, if $Y \subset X$ is a $\Com_T$-invariant subset of $X$ and $x \in Y$ then $\Com_T(x) = \left\{ Sx : S \in \Com_T \right\} \subset \bigcup_{S \in \Com_T} S(Y) \subset Y$.
\end{proof}

When $(X,T)$ is a non-linear system then $\Com_T(x)$ is still the {\em smallest $\Com_T$-invariant subset of $X$ containing the point $x \in X$}. Let us now generalize Corollary~\ref{Cor:return-sets-commute} to direct products by using the following notation: given a system $(X,T)$ and a subset $Y \subset X$ we will denote by
\[
\Nc_{pr}(Y) := \{ A \subset \NN_0 : N(x,U) \subset A \text{ for some } x \in Y \text{ and some neighbourhood } U \text{ of } x \},
\]
the {\em family of pointwise-recurrent return sets} of the points of $Y$. Note that $\Nc_{pr}(\{x\})$ is a filter if and only if $x \in \Rec(T)$: indeed, if $x \notin \Rec(T)$ then $\Nc_{pr}(\{x\})=\Part(\NN_0)$, which is not a filter; conversely, for any pair $U$ and $V$ of neighbourhoods for a recurrent vector $x \in \Rec(T)$ then we have that $N(x,U\cap V) \subset N(x,U) \cap N(x,V)$. Note also that a dynamical system $(X,T)$ is $\Fc$-recurrent if there exits a dense set $Y \subset X$ such that $\Nc_{pr}(Y) \subset \Fc$.

\begin{theorem}\label{The:F-rec-commute}
	Let $(X,T)$ be a (linear) dynamical system and let $\Fc$ be a Furstenberg family. Given $x \in \Fc\Rec(T)$ we have that
	\[
	\Nc_{pr}(\Com_T(x)) \ = \ \Nc_{pr}(\{x\}),
	\]
	and hence $\Com_T(x)$ is a $\Com_T$-invariant (vector) subspace of $X$ with the property that
	\[
	\Com_T(x)^N \ \subset \ \Fc\Rec(T_{(N)}) \quad \text{ for every } N \in \NN.
	\]
\end{theorem}
\begin{proof}
	Obviously $\Nc_{pr}(\{x\}) \subset \Nc_{pr}(\Com_T(x))$. The inclusion $\Nc_{pr}(\Com_T(x)) \subset \Nc_{pr}(\{x\})$ follows from Corollary~\ref{Cor:return-sets-commute}: given $y \in \Com_T(x)$ and any neighbourhood $U$ of $y$ there is a neighbourhood $V$ of $x$ such that
	\[
	N_T(y,U) = N_T(x,V) \in \Nc_{pr}(\{x\}).
	\]
	The rest of the result follows from the filter condition satisfied by the family $\Nc_{pr}(\{x\})$: indeed, given any $N \in \NN$ and a set $\{x_1,x_2,...,x_N\} \subset \Com_T(x)$ write
	\[
	z := (x_1,x_2,...,x_N) \in X^N.
	\]
	Then, given any neighbourhood $U \subset X^N$ of $z$ we can find neighbourhoods $U_i \subset X$ of $x_i$, for $1\leq i\leq N$, such that $U_1\times\cdots\times U_N \subset U$, and hence
	\[
	N_{T_{(N)}}(z,U) \supset \bigcap_{i=1}^N N_T(x_i,U_i) \in \Nc_{pr}(\{x\}) \subset \Fc.
	\]
	The arbitrariness of $U$ implies that $z \in \Fc\Rec(T_{(N)})$. Finally, if $(X,T)$ is a linear dynamical system, then the set $\Com_T(x)$ is a $\Com_T$-invariant vector subspace of $X$ by Lemma~\ref{Lem:Com_T-subspace}.
\end{proof}

\begin{remark}
	Theorem~\ref{The:F-rec-commute} is an extension of Corollary~\ref{Cor:return-sets-commute} to every $N$-fold direct product of a given dynamical system. As we have already mentioned, these kind of arguments were already used in \cite[Theorem~9.1]{CoMaPa2014} for usual recurrence and operators on Banach spaces, and in \cite[Proposition~3.8]{CarMur2022_arXiv} for arbitrary Furstenberg families.
\end{remark}

\section{Infinite-dimensional vector spaces in $\Fc\Rec(T)$}\label{Sec:5lin+dense.lin}

The objective of this section is to study, for a linear dynamical system $(X,T)$ and any arbitrary Furstenberg family $\Fc$, when the set of $\Fc$-recurrent vectors is {\em lineable} or {\em dense lineable}, i.e.\ to establish whether it contains a (possibly dense) infinite-dimensional vector subspace. Theorem~\ref{The:F-rec-commute} will be our main tool for this study, which motivation stems from the following two facts already shown in previous sections:
\begin{enumerate}[--]
	\item quasi-rigidity coincides with $\Fc(A)$-recurrence (see Proposition~\ref{Pro:qr.filter});
	
	\item quasi-rigidity of an operator $T$ implies that $\Rec(T)$ is dense lineable (see Proposition~\ref{Pro:qr->dense.lin}).
\end{enumerate}
It is then natural to ask whether the set $\Fc\Rec(T)$ contains an infinite-dimensional vector space for other Furstenberg families $\Fc$. We show that $\Fc\Rec(T)$ is lineable as soon as $T$ is $\Fc$-recurrent (see Theorem~\ref{The:lineability} and Corollary~\ref{Cor:lineability} below), and we obtain some (natural) sufficient conditions implying that $\Fc\Rec(T)$ is dense lineable (see Theorem~\ref{The:dense.lineability}), obtaining then the {\em Herrero-Bourdon theorem} for $\Fc$-hypercyclicity (see Subsection~\ref{SubSec:5.3dense.lin.F-hyp}).

\subsection{Lineability}

Given a vector space $X$ and a subset of vectors $Y \subset X$ with some property, we say that $Y$ is {\em lineable} if there exists an infinite-dimensional vector subspace $Z \subset X$ such that $Z\setminus\{0\} \subset Y$. In our case, given a linear dynamical system $(X,T)$ we always have that $0 \in \Fc\Rec(T)$ for every Furstenberg family $\Fc$, so the set of $\Fc$-recurrent vectors will be lineable if it admits an infinite-dimensional vector space (which includes the zero-vector).\\[-5pt]

In order to prove that $\Fc\Rec(T)$ is lineable as soon as $(X,T)$ is $\Fc$-recurrent, we will observe that the (span of the) {\em unimodular eigenvectors} are the only recurrent vectors whose orbits have a finite-dimensional linear span (see Lemma~\ref{Lem:finite.dimension} below). Recall that:

\begin{definition}
	Given a \textbf{complex-linear} dynamical system $T:X\longrightarrow X$, a vector $x \in X$ is called a {\em unimodular eigenvector} for $T$ if $x \neq 0$ and $Tx=\lambda x$ for some unimodular complex number $\lambda \in \TT$. We denote by $\Ec(T)$ the {\em set of unimodular eigenvectors} for $T$, i.e.\
	\[
	\Ec(T) = \{ x \in X\setminus\{0\} : Tx = \lambda x \text{ for some } \lambda \in \TT \}.
	\]
\end{definition}
	
Every finite linear combination of unimodular eigenvectors is a $\Del^*$-recurrent vector (see for instance \cite[Proposition 4.1]{GriLo2023}) and the following holds (see \cite[Proposition~2.33]{GrPe2011_book}):
\[
\Per(T) = \lspan\{ x \in X : Tx = e^{\alpha \pi i}x \text{ for some } \alpha \in \QQ \} \subset \lspan(\Ec(T)) \subset \Del^*\Rec(T).
\]
In order to treat both real and complex cases at the same time, we need a set of (\textbf{real}) vectors having an analogous recurrent-behaviour to that of unimodular eigenvectors:

\begin{remark}[\textbf{A well-known conjugacy}]\label{Rem:complexification}
	The {\em complexification} $(\til{X},\til{T})$ of a \textbf{real-linear} system $T:X\longrightarrow X$ is defined in the following way (see \cite{MuSaTo1999,MoMuPeSe2022} and \cite[Exercise 2.2.7]{GrPe2011_book}):
	\begin{enumerate}[--]
		\item the space $\til{X} := \left\{ x + iy : x,y \in X \right\}$, which is topologically identified with $X\oplus X$, and becomes a \textbf{complex} F-space when endowed with the multiplication by complex scalars $(\alpha+i\beta)(x+iy) = (\alpha x - \beta y) + i(\alpha y + \beta x)$ for every $\alpha,\beta \in \RR$ and every $x,y \in X$;
		
		\item and the operator $\til{T}:\til{X}\longrightarrow\til{X}$ is defined as $\til{T}(x+iy) = Tx + iTy$ for every $x,y \in X$. It is a continuous \textbf{complex-linear} operator acting on $\til{X}$.
	\end{enumerate}
	Defining the map $J:X\oplus X\longrightarrow\til{X}$ as $J(x,y):=x+iy \in \til{X}$ for every $(x,y) \in X\oplus X$, the diagram
	\begin{equation*}
		\begin{CD}
			X\oplus X  @>{T\oplus T}>> X\oplus X \\
			@V{J}VV  @VV{J}V \\
			\til{X}  @>{\til{T}}>>  \til{X} \\
		\end{CD}
	\end{equation*}
	commutes so $J$ is a conjugacy (see Definition~\ref{Def:homeo.quasi-con.con}). In this setting we define the (\textbf{real}) {\em set of unimodular eigenvectors} for $T$ as
	\[
	\Ec(T) := \left\{ x \in X : \text{ there exists } y \in X \text{ such that } x+iy \in \Ec(\til{T}) \right\}.
	\]
	These \textbf{real} unimodular eigenvectors have the same recurrent-behaviour than the \textbf{complex} unimodular eigenvectors. In fact, the following properties are easily checked:
	\[
	\lspan(\Ec(T)) = \left\{ x \in X : \text{ there exists } y \in X \text{ such that } x+iy \in \lspan(\Ec(\til{T})) \right\},
	\]
	\[
	\Per(T) \subset \lspan(\Ec(T)) \subset \Del^*\Rec(T).
	\]
\end{remark}

\begin{lemma}\label{Lem:finite.dimension}
	Let $(X,T)$ be a (real or complex) linear dynamical system. For each recurrent vector $x \in \Rec(T)$ the following statements are equivalent:
	\begin{enumerate}[{\em(i)}]
		\item $x \in \lspan(\Ec(T))$;
		
		\item $\dim \left( \lspan\{T^nx : n \in \NN_0\} \right) < \infty$.
	\end{enumerate}
	In particular, given any Furstenberg family $\Fc$ for which $\Fc\Rec(T)\setminus\lspan(\Ec(T))\neq\varnothing$, there exists an infinite-dimensional $T$-invariant vector subspace $Z \subset X$ with the property that $Z^N \subset \Fc\Rec(T_{(N)})$ for every $N \in \NN$.
\end{lemma}
\begin{proof}
	For (i) $\Rightarrow$ (ii) compute the orbit of a vector from $\lspan(\Ec(T))$ in both real and complex cases, and observe that its span is finite-dimensional. For (ii) $\Rightarrow$ (i) let $x \in \Rec(T)$ and suppose that the $T$-invariant subspace $E=\lspan\{ T^nx : n \in \NN_0 \}$ is finite-dimensional:
	\begin{enumerate}[(1)]
		\item If $(X,T)$ is complex, then $T\res_E:E\longrightarrow E$ is a recurrent complex-linear operator on a finite-dimensional space. By \cite[Theorem 4.1]{CoMaPa2014} there exists a basis of $E$ formed by unimodular eigenvectors for $T\res_E$ (and hence for $T$) so $x \in E \subset \lspan(\Ec(T))$.
		
		\item If $(X,T)$ is real, then we can identify the complexification $\til{T\res_E}:\til{E}\longrightarrow\til{E}$ with the direct sum $T\res_E\oplus T\res_E$ (see Remark~\ref{Rem:complexification}). Theorem~\ref{The:F-rec-commute} implies that
		\[
		E\oplus E \subset \Rec(T\res_E\oplus T\res_E),
		\]
		so $\til{T\res_E}$ is a recurrent complex-linear operator on a finite-dimensional space. As in case (1) above, by \cite[Theorem 4.1]{CoMaPa2014} there exists a basis of $\til{E}$ formed by unimodular eigenvectors for $\til{T}$, so $x+i0 \in \til{E} \subset \lspan(\Ec(\til{T}))$ and hence $x \in \lspan(\Ec(T))$.
	\end{enumerate}
	Finally, if $\Fc$ is a Furstenberg family for which there exists some $x\in\Fc\Rec(T)\setminus\lspan(\Ec(T))$, the equivalence (i) $\Leftrightarrow$ (ii) implies that $Z:=\lspan\{ T^nx : n \in \NN_0 \}$ is an infinite-dimensional $T$-invariant vector subspace of $X$. Theorem~\ref{The:F-rec-commute} then shows
	\[
	Z^N \subset \Com_T(x)^N \subset \Fc\Rec(T_{(N)}) \quad \text{ for every } N \in \NN.\qedhere
	\] 
\end{proof}

The strongest $\Fc$-recurrence notion considered in this paper and fulfilling the condition that $\lspan(\Ec(T)) \subset \Fc\Rec(T)$ for every $T \in \Lc(X)$ is that of $\Del^*$-recurrence. The properties of ``{\em having a spanning set of unimodular eigenvectors}'' and that of ``{\em being $\Delta^*$-recurrent}'' have been deeply related in the recent work \cite{GriLo2023} and they are specially near when one considers power-bounded operators (see \cite[Theorem~1.9]{GriLo2023}). It is then natural to ask if the equality $\lspan(\Ec(T)) = \Del^*\Rec(T)$ holds for some general class of linear dynamical systems $(X,T)$, a natural candidate being that of power-bounded operators. The answer is negative as we show in the following trivial example by using Lemma~\ref{Lem:finite.dimension}:

\begin{example}
	{\em There exists a linear dynamical system admitting a $\Del^*$-recurrent vector such that the linear span of its orbit is infinite-dimensional}: consider $X=c_0(\NN)$ or $\ell^p(\NN)$ for some $1\leq p<\infty$ with their usual norms. Define a linear map $T:c_{00}(\NN)\longrightarrow c_{00}(\NN)$ in the following way: $Te_1=e_1$ and for every $k\geq 2$
	\begin{equation*}
		Te_k =  \begin{cases}
					e_{k+1}, & \text{if } 2^m < k < 2^{m+1}, \\
					e_{2^m+1}, & \text{if } k=2^{m+1},
				\end{cases}
	\end{equation*}
	for each $m\geq 0$, where $e_k=(\del_{k,n})_{n=1}^{\infty}$ is the $k$-th vector of the canonical basis of $X$. Since $\|Tx\| = \|x\|$ for each $x \in c_{00}(\NN)$, $T$ extends to a {\em linear isometry} on the whole space $X$. We now show that the vector $x = \sum_{m\geq 0} \frac{1}{2^m} e_{2^m+1} \in X$ has the required properties:
	\begin{enumerate}[1)]
		\item $x \in \Del^*\Rec(T)$: given any $\eps>0$ there is $m_{\eps} \in \NN$ such that $\left\| \sum_{m > m_{\eps}} \frac{1}{2^m}e_{2^m+1} \right\| < \frac{\eps}{2}$. Then, given $n \in 2^{m_{\eps}}\cdot\NN_0$ we have that
		\[
		T^n\left( \sum_{m \leq m_{\eps}} \frac{1}{2^m}e_{2^m+1} \right) = \sum_{m \leq m_{\eps}} \frac{1}{2^m}e_{2^m+1} \quad \text{ so } \quad \|T^nx - x\| \leq 2 \left\| \sum_{m > m_{\eps}} \frac{1}{2^m}e_{2^m+1} \right\| < \eps.
		\] 
		Hence $2^{m_{\eps}}\cdot\NN_0 \subset N_T(x,B(x,\eps))$, so $x \in \Del^*\Rec(T)$ since $2^{m}\cdot\NN_0 \in \Del^*$ for all $m\in\NN$.
		
		\item $\lspan\{T^nx : n\in\NN_0\}$ is infinite-dimensional: otherwise there would exist a polynomial
		\[
		p(z) = \sum_{i=1}^N a_i z^i \text{ with } a_i \in \KK \text{ and } a_N\neq 0 \quad \text{ such that } \quad p(T)x=0,
		\]
		and taking $m_0 \in \NN$ such that $N<2^{m_0}$ we would arrive to the contradiction
		\[
		0 = \left[p(T)x\right]_{2^{m_0}+1+N} = \left[ a_N \cdot T^Nx \right]_{2^{m_0}+1+N} = \frac{a_N}{2^{m_0}} \neq 0.
		\]
	\end{enumerate}
\end{example}

We finally state the desired {\em lineability} results:

\begin{theorem}[\textbf{Lineability}]\label{The:lineability}
	Let $\Fc$ be a Furstenberg family such that the inclusion $\lspan(\Ec(S)) \subset \Fc\Rec(S)$ holds for every linear dynamical system $(Y,S)$. Then, for each $\Fc$-recurrent linear system $(X,T)$ there exists an infinite-dimensional $T$-invariant vector subspace $Z \subset X$ with the property that $Z^N \subset \Fc\Rec(T_{(N)})$ for every $N \in \NN$.
\end{theorem}
\begin{proof} 
	By assumption $\Fc\Rec(T)$ is dense. We distinguish two cases:
	\begin{enumerate}[(1)]
		\item If $\lspan(\Ec(T))$ is finite-dimensional (or simply if it is not dense in $X$), then we have that $\Fc\Rec(T)\setminus\lspan(\Ec(T))\neq\varnothing$ and Lemma~\ref{Lem:finite.dimension} yields the desired vector subspace.
		
		\item If $\lspan(\Ec(T))$ is infinite-dimensional, then $(\lspan(\Ec(T)))^N = \lspan(\Ec(T_{(N)})) \subset \Fc\Rec(T_{(N)})$ for every $N \in \NN$, and $Z:=\lspan(\Ec(T))$ is the required vector subspace.\qedhere
	\end{enumerate}
\end{proof}

The previous theorem is ``optimal'' for the families considered in this paper: we always have the inclusions
\[
\Per(T) \subset \lspan(\Ec(T)) \subset \Del^*\Rec(T),
\]
and we observed in Section~\ref{Sec:4F-rec} that the families $\Fc$ considered in this paper for which $\Fc\Rec(T) \subset \Del^*\Rec(T)$ are such that $\Fc\Rec(T) \subset \Per(T)$ for every system $(X,T)$. We deduce (at least) the following result:

\begin{corollary}\label{Cor:lineability}
	Let $(X,T)$ be a linear dynamical system. If the operator $T$ is recurrent (resp.\ $\AP$, reiterative, $\Uc$-frequent, frequent, uniformly, $\IP^*$ or $\Del^*$-recurrent), then $\Rec(T)$ (resp.\ $\AP\Rec(T)$, $\RRec(T)$, $\UFRec(T)$, $\FRec(T)$, $\URec(T)$, $\IP^*\Rec(T)$ or $\Del^*\Rec(T)$) contains an infinite-dimensional $T$-invariant vector space, and in particular it is lineable.
\end{corollary}

Theorem~\ref{The:lineability} (and hence Corollary~\ref{Cor:lineability}) is still true with the same proof if we replace the original assumption that ``{\em $T$ is $\Fc$-recurrent}'' by the very much less restrictive hypothesis that ``{\em $\Fc\Rec(T)$ is dense in some infinite-dimensional closed subspace $Y \subset X$}'', or even by the formally weaker ``{\em $\Fc\Rec(T)$ spans an infinite-dimensional vector subspace}''. These are also necessary conditions for the lineability property as the following example shows:

\begin{example}
	Let $X=c_0(\NN)$ or $\ell^p(\NN)$ for some $1\leq p<\infty$ with their usual norm. Consider a bounded sequence $(\lambda_n)_{n\in\NN} \subset \CC$ and the respective multiplication operator defined as
	\[
	T\left( (x_n)_{n\in\NN} \right) := (\lambda_n \cdot x_n)_{n\in\NN} \quad \text{ for each } (x_n)_{n\in\NN} \in X. 
	\]
	For any fixed any $N \in \NN$, consider a sequence $(\lambda_n)_{n\in\NN}$ with the properties that
	\[
	\lambda_n=1 \text{ for } n\leq N \qquad \text{ and } \qquad \lambda_n \notin \TT \text{ for } n>N.
	\]
	Then it is trivial to check that, for any Furstenberg family $\Fc$, the set $\Fc\Rec(T)$ is exactly a finite-dimensional vector space of dimension $N$.
\end{example}

\subsection{Dense Lineability}

Given a topological vector space $X$ and a subset of vectors $Y \subset X$ with some property, we say that $Y$ is {\em dense lineable} if there exists a dense infinite-dimensional vector subspace $Z \subset X$ such that $Z\setminus\{0\} \subset Y$. As it happens for {\em lineability}, given a linear dynamical system $(X,T)$ we always have that $0 \in \Fc\Rec(T)$ for every Furstenberg family $\Fc$, so the set of $\Fc$-recurrent vectors will be dense lineable if it admits a dense infinite-dimensional vector space (which includes the zero-vector).\\[-5pt]

Given a linear dynamical system $(X,T)$ we provide some sufficient conditions for the set $\Fc\Rec(T)$ to be dense lineable. Filters will play a fundamental role in our study. The motivation for our main result (see Theorem~\ref{The:dense.lineability} below) is again the notion of quasi-rigidity:
\begin{enumerate}[--]
	\item quasi-rigidity implies that $\Rec(T)=\Ic\Rec(T)$ is dense lineable (see Proposition~\ref{Pro:qr->dense.lin});
	
	\item quasi-rigidity coincides with $\Fc(A)$-recurrence, and $\Fc(A)$ is a filter (see Proposition~\ref{Pro:qr.filter});
	
	\item i.e.\ {\em there is a dense set $Y \subset X$ and a filter $\Fc'=\Fc(A)$ such that $\Nc_{pr}(Y) \subset \Fc' \subset \Ic$}.
\end{enumerate}
We can generalize these ideas to arbitrary Furstenberg families $\Fc$, finding some filter $\Fc'$ contained in $\Fc$. In particular, we extend \cite[Theorem 6.1]{BoGrLoPe2022} where it is shown that the set of $\IP^*$-recurrent vectors is always a vector subspace:

\begin{theorem}[\textbf{Dense Lineability}]\label{The:dense.lineability}
	Let $(X,T)$ be a linear dynamical system and let $\Fc$ be a Furstenberg family. If there exist a dense subset $Y \subset X$ and a filter $\Fc'$ such that
	\[
	\Nc_{pr}(Y) \subset \Fc' \subset \Fc,
	\]
	then \ $Z:=\lspan( \bigcup_{S\in\Com_T} S(Y) ) \subset X$ \ is a dense infinite-dimensional $T$-invariant vector subspace with the property that $Z^N \subset \Fc\Rec(T_{(N)})$ for every $N \in \NN$. In particular, if $\Fc$ is a filter itself, then for every $N \in \NN$ we have the equality
	\[
	\Fc\Rec(T)^N = \Fc\Rec(T_{(N)}),
	\]
	which is a $T$-invariant vector subspace of $X^N$, and the following are equivalent:
	\begin{enumerate}[{\em(i)}]
		\item $T$ is $\Fc$-recurrent;
		
		\item $\Fc\Rec(T_{(N)})$ is a dense infinite-dimensional vector subspace of $X^N$ for every $N \in \NN$.
	\end{enumerate}
\end{theorem}
\begin{proof}
	It is clear that $Z \subset X$ is a dense infinite-dimensional $T$-invariant vector subspace. Let us show that $\Nc_{pr}(Z) \subset \Fc'$: given $y_1,y_2,...,y_N \in \bigcup_{S\in\Com_T} S(Y)$ and any neighbourhood $U \subset X$ of $x:=\sum_{i=1}^N y_i \in Z$ we can find neighbourhoods $U_i$ of $y_i$, for $1\leq i\leq N$, such that
	\[
	\sum_{i=1}^N U_i \subset U \subset X \quad \text{ and hence } \quad N_{T}(x,U) \supset \bigcap_{i=1}^N N_T(y_i,U_i) \in \Fc' \subset \Fc,
	\]
	since $\Nc_{pr}(\{y_i\}) \subset \Fc'$ for all $1\leq i\leq N$ by Theorem~\ref{The:F-rec-commute}. Finally, given any $N \in \NN$, $x_1,x_2,...,x_N \in Z$ and any neighbourhood $V \subset X^N$ of the $N$-tuple $z:=(x_1,...,x_N) \in X^N$, we can find neighbourhoods $V_i$ of $x_i$, for $1\leq i\leq N$, such that
	\[
	V_1\oplus\cdots\oplus V_N \subset V \subset X^N \quad \text{ and hence } \quad  N_{T_{(N)}}(z,V) \supset \bigcap_{i=1}^N N_T(x_i,V_i) \in \Fc' \subset \Fc.
	\]
	The arbitrariness of $V \subset X^N$ implies that $z \in \Fc\Rec(T_{(N)})$.
\end{proof}

Note that Theorem~\ref{The:dense.lineability} extends the result obtained in Proposition~\ref{Pro:qr->dense.lin} to every $N$-fold direct sum operator. Moreover, in view of Theorems~\ref{The:F-rec-commute} and \ref{The:dense.lineability} we can generalize Proposition~\ref{Pro:sufficient-quasi-rigid} establishing, for dense lineability, the following (natural) sufficient conditions:

\begin{proposition}\label{Pro:dense.lineability}
	Let $(X,T)$ be a linear dynamical system and let $\Fc$ be a Furstenberg family. If $T$ admits an $\Fc$-recurrent vector $x \in \Fc\Rec(T)$ with a dense $\Com_T$-orbit
	\[
	\Com_T(x) = \{Sx : S \in \Com_T \},
	\]
	then $\Fc\Rec(T_{(N)})$ is dense lineable in $X^N$ for every $N \in \NN$. In particular, the latter is true whenever any of the following holds:
	\begin{enumerate}[--]
		\item $T$ admits an $\Fc$-recurrent and cyclic vector;
		
		\item $T$ admits an $\Fc$-hypercyclic vector.
	\end{enumerate}
\end{proposition}
\begin{proof}
	Given a vector $x \in \Fc\Rec(T)$ for which $\Com_T(x)$ is dense in $X$, Theorem~\ref{The:F-rec-commute} implies that $\Nc_{pr}(\Com_T(x)) \subset \Nc_{pr}(\{x\}) \subset \Fc$ and the result follows from Theorem~\ref{The:dense.lineability}.
\end{proof}

The dense lineability of the set $\Fc\Rec(T_{(N)})$ implies that $T_{(N)}$ is $\Fc$-recurrent. A slightly weaker hypothesis allows us to keep the $\Fc$-recurrence of every $N$-fold direct product:

\begin{proposition}[\textbf{$\Fc$-recurrence for all $(X^N,T_{(N)})$}]\label{Pro:prod-F-recurrence}
	Let $(X,T)$ be a linear dynamical system and let $\Fc$ be a Furstenberg family. If $T$ is $\Fc$-recurrent and if there exists a vector $x \in X$ with a dense $\Com_T$-orbit
	\[
	\Com_T(x) = \{Sx : S \in \Com_T \},
	\]
	then $T_{(N)}:X^N\longrightarrow X^N$ is $\Fc$-recurrent for every $N \in \NN$. In particular, the latter is true whenever any of the following holds:
	\begin{enumerate}[--]
		\item $T$ is $\Fc$-recurrent and cyclic;
		
		\item $T$ is $\Fc$-recurrent and hypercyclic.
	\end{enumerate}
\end{proposition}
This has been recently proved in \cite[Proposition~3.8]{CarMur2022_arXiv}, but we repeat here the argument for the sake of completeness:
\begin{proof}
	Fix $N \in \NN$. By Theorem~\ref{The:F-rec-commute} we have the inclusion
	\[
	Y := \bigcup_{y \in \Fc\Rec(T)} \Com_T(y)^N \subset \Fc\Rec(T_{(N)}),
	\]
	so it is enough to show that $Y$ is dense in $X^N$. By assumption the point $x \in X$ has a dense $\Com_T$-orbit, so given any non-empty open subset $U \subset X^N$ there are $N$ continuous maps $S_1,S_2,...,S_N \in \Com_T$ such that $(S_1x,S_2x,...,S_Nx) \in U$. By the continuity of $S_1,S_2,...,S_N$, and since $T$ is $\Fc$-recurrent, we can find $y \in \Fc\Rec(T)$ near enough to $x$ such that
	\[
	(S_1y,S_2y,...,S_Ny) \in \ \Com_T(y)^N \cap U \ \subset \ Y \cap U.\qedhere
	\]
\end{proof}

\begin{remark}
	We include here some comments about the previous results:
	\begin{enumerate}[(a)]
		\item Proposition~\ref{Pro:prod-F-recurrence} remains valid for non-linear dynamical systems. It is a slight improvement of \cite[Theorem 9.1]{CoMaPa2014} in terms of Furstenberg families and non-linear commuting maps, and it has been independently proved in the recent work \cite[Proposition~3.8]{CarMur2022_arXiv}.
		
		\item Given a Furstenberg family $\Fc$, the following is a natural open question:
		\begin{problem}
			Suppose that a linear dynamical system $(X,T)$ does not admit any dense $\Com_T$-orbit. Can every $N$-fold direct sum $(X^N,T_{(N)})$ be $\Fc$-recurrent?
		\end{problem}
		
		\item The assumptions of Proposition~\ref{Pro:dense.lineability} imply those of Proposition~\ref{Pro:prod-F-recurrence}. Conversely:
		\begin{problem}
			Let $\Fc$ be a Furstenberg family. Can an $\Fc$-recurrent linear system $(X,T)$ admit dense $\Com_T$-orbits but fulfill that $\Fc\Rec(T) \cap \{ x \in X : \cl{\Com_T(x)} = X \} = \varnothing$?
		\end{problem}
		This seems to be a tough question since the commutator of an operator is usually difficult to describe. If instead of dense $\Com_T$-orbits we just consider the set of {\em hypercyclic} vectors, then the answer is yes: Menet constructed a {\em chaotic} (and hence with dense periodic vectors) operator $T$, which is not $\Uc$-frequently hypercyclic (see \cite{Menet2017}), i.e.
		\begin{enumerate}[--]
			\item {\em Such an operator $T$ fulfills the assumptions of Proposition~\ref{Pro:prod-F-recurrence} for every Furstenberg family $\Fc$ satisfying that $\Delta^* \subset \Fc \subset \Dsup$, but we have that $\Fc\Rec(T) \cap \HC(T) = \varnothing$ so that $T$ has no vector which is $\Fc$-recurrent and hypercyclic at the same time}.
		\end{enumerate}
		Considering now the set of {\em cyclic} vectors (instead of dense $\Com_T$-orbits or hypercyclic vectors), Corollary~\ref{Cor:Rec+Cyc} and Proposition~\ref{Pro:dense.lineability} imply that:
		\begin{enumerate}[--]
			\item {\em If $(X,T)$ is a recurrent and cyclic linear dynamical system, then the set $\Rec(T_{(N)})$ is dense lineable for every $N \in \NN$}.
		\end{enumerate}
		In general we cannot change $\Rec(T_{(N)})$ into $\Fc\Rec(T_{(N)})$ unless we could ensure that $\Fc\Rec(T)$ is a co-meager set. This is the case for $\AP$-recurrence (see \cite{KwiLiOpYe2017,CarMur2022_MS}):
		\begin{enumerate}[--]
			\item {\em If $(X,T)$ is an $\AP$-recurrent (i.e.\ multiple recurrent) and cyclic linear dynamical system, then the set $\AP\Rec(T_{(N)})$ is dense lineable for every $N \in \NN$}.
		\end{enumerate}		
		In \cite[Example~2.4]{BoGrLoPe2022} it is exhibited a $\BDsup$-recurrent (i.e.\ reiteratively recurrent) operator $T$ for which the set $\BDsup\Rec(T)=\RRec(T)$ is meager. Such an operator $T$ is not cyclic, and in view of \cite[Theorem~2.1]{BoGrLoPe2022} (result which states that: {\em the set of $\BDsup$-recurrent vectors is co-meager for every $\BDsup$-recurrent and hypercyclic operator}) we ask:
		\begin{problem}\footnote{This problem has recently been solved in the negative; see \cite{LoMe2024_JMAA}.}
				Let $(X,T)$ be a $\BDsup$-recurrent (i.e.\ reiteratively recurrent) and cyclic linear dynamical system. Is $\RRec(T)$ a co-meager set?
		\end{problem}
	\end{enumerate}
\end{remark}

\subsection{Dense Lineability for $\Fc$-hypercyclicity}\label{SubSec:5.3dense.lin.F-hyp}

Under some natural conditions on the Furstenberg family $\Fc$ we can apply the $\Fc$-recurrence theory developed above in order to obtain some results regarding $\Fc$-hypercyclicity. We will consider {\em right-invariant} and {\em upper} Furstenberg families. Let us recall the definitions:\\[-5pt]

A family $\Fc$ is said to be {\em right-invariant} if for every $A \in \Fc$ and every $n \in \NN_0$ the set
\[
A+n = \{k+n : k \in A\}
\]
also belongs to $\Fc$. For example, the families $\Ic, \AP, \Sc, \Tc, \TS, \Ic^*$ and the density ones are easily seen to be right-invariant. However, the families $\IP$ and $\Del$ together with their dual families $\IP^*$ and $\Del^*$ are not right-invariant (see \cite[Proposition~5.6]{BMPP2019}).\\[-5pt]

It is shown in \cite{BoGrLoPe2022} that given any right-invariant Furstenberg family $\Fc$, then every single $\Fc$-recurrent and hypercyclic vector is indeed $\Fc$-hypercyclic. As a consequence we obtain a kind of {\em Herrero-Bourdon theorem} (see~\cite[Theorem 2.55]{GrPe2011_book}) for $\Fc$-hypercyclicity:

\begin{corollary}
	Let $(X,T)$ be a linear dynamical system and assume that $\Fc$ is a right-invariant Furstenberg family. If $x$ is an $\Fc$-hypercyclic vector for $T$, then
	\[
	\{ p(T)x : p \text{ polynomial} \}\setminus\{0\},
	\]
	is a dense set of $\Fc$-hypercyclic vectors. In particular, any $\Fc$-hypercyclic operator admits a dense invariant vector space consisting, except for zero, of $\Fc$-hypercyclic vectors.
\end{corollary}
\begin{proof}
	Follows from the original Herrero-Bourdon theorem, together with Theorem~\ref{The:F-rec-commute} and the equality
	\[
	\Fc\HC(T) = \Fc\Rec(T) \cap \HC(T),
	\]
	showed in \cite{BoGrLoPe2022} for every right-invariant Furstenberg family $\Fc$.
\end{proof}

Apart from {\em right-invariance} we may assume the {\em u.f.i.\ upper} condition (defined in \cite{BG2018}) to obtain even stronger results: a family $\Fc$ is said to be {\em upper} if it can be written as
\[
\Fc = \bigcup_{\delta \in D} \Fc_{\delta} \quad \text{ with } \quad \Fc_{\delta} = \bigcap_{m \in M} \Fc_{\delta,m}
\]
for sets $\Fc_{\delta,m} \subset \Part(\NN_0)$ such that $\delta\in D$ and $m \in M$, where $D$ is arbitrary, $M$ is countable and they have the following properties:
\begin{enumerate}
	\item[(i)] for any $\Fc_{\delta,m}$ and any $A \in \Fc_{\delta,m}$ there exists a finite set $F \subset \NN_0$ such that
	\[
	A\cap F \subset B \text{ implies } B \in \Fc_{\delta,m}; 
	\]
	
	\item[(ii)] for any $A \in \Fc$ there exists some $\delta \in D$ such that for all $n \in \NN_0$ we have
	\[
	(A-n)\cap\NN_0 \in \Fc_{\delta}.
	\]
\end{enumerate}
We say that an upper family $\Fc$ as above is {\em uniformly finitely invariant} ({\em u.f.i.}\ for short) if for any $A \in \Fc$ there is some $\delta \in D$ such that, for all $n\in\NN$, $A \cap [n,\infty[ \in \Fc_{\delta}$.\\[-5pt]

The families $\Ic,\AP,\BDsup$ and $\Dsup$ are easily checked to be {\em u.f.i.\ upper}, while $\Dinf$ is not even {\em upper} (see \cite{BG2018,CarMur2022_MS}). As it is shown in \cite{BG2018}, there exists a kind of {\em Birkhoff transitivity theorem} for these {\em u.f.i.\ upper} families, and in particular, the following holds:

\begin{corollary}
	Let $(X,T)$ be an $\Fc$-recurrent linear dynamical system where $\Fc$ is a u.f.i.\ upper Furstenberg family. If there is a dense set $X_0 \subset X$ such that $T^nx\to 0$ for each $x \in X_0$, then the set $\Fc\Rec(T_{(N)})$ is dense lineable for every $N \in \NN$. In particular:
	\begin{enumerate}[{\em(a)}]
		\item If $T$ is recurrent then $\Rec(T_{(N)})$ is dense lineable for every $N \in \NN$;
		
		\item If $T$ is $\AP$-recurrent then $\AP\Rec(T_{(N)})$ is dense lineable for every $N \in \NN$;
		
		\item If $T$ is reiteratively recurrent then $\RRec(T_{(N)})$ is dense lineable for every $N \in \NN$;
		
		\item If $T$ is $\Uc$-frequently recurrent then $\UFRec(T_{(N)})$ is dense lineable for every $N \in \NN$.
	\end{enumerate}
\end{corollary}
\begin{proof}
	By \cite[Theorem~8.5]{BoGrLoPe2022}, if $T$ is $\Fc$-recurrent for a u.f.i.\ upper family $\Fc$ and if there exists a dense set of vectors $x \in X$ such that $T^nx \to 0$ as $n \to \infty$, then $T$ is $\Fc$-hypercyclic. By Proposition~\ref{Pro:dense.lineability} the set $\Fc\Rec(T_{(N)})$ is then dense lineable for every $N \in \NN$. The particular cases follow from the fact that $\Ic$, $\AP$, $\BDsup$ and $\Dsup$ are u.f.i.\ upper families.
\end{proof}

Moreover, if we consider the Furstenberg family $\BDsup$ of {\em positive upper Banach density sets} we can easily reprove a particular case of \cite[Theorem 2.5 and Corollary 2.8]{ErEsMe2021}:

\begin{corollary}[\textbf{\cite[Theorem~2.5 and Corollary~2.8]{ErEsMe2021}}]
	Let $(X,T)$ be a reiteratively hypercyclic linear system. Then the $N$-fold direct sum operator $T_{(N)}:X^N\longrightarrow X^N$ is reiteratively hypercyclic for every $N \in \NN$.
\end{corollary}
\begin{proof}
	By \cite[Proposition 4]{BMPP2016} we have that $T_{(N)}$ is hypercyclic. By Proposition~\ref{Pro:prod-F-recurrence} we have that $T_{(N)}$ is reiteratively recurrent. By \cite[Theorem 2.1]{BoGrLoPe2022} every hypercyclic vector of $T_{(N)}$ is also reiteratively hypercyclic, so that $T_{(N)}$ is reiteratively hypercyclic.
\end{proof}

\section{Open problems}\label{Sec:6open}

In this section we include three open questions and a few comments related to them. They essentially ask ``what happens if we cannot apply Theorem~\ref{The:dense.lineability}'': 

\begin{problems}
	Let $\Fc$ be a Furstenberg family and assume that it is not a filter. Suppose that $T \in \Lc(X)$ is a continuous linear $\Fc$-recurrent operator:
	\begin{enumerate}[(a)]
		\item Is the set $\Fc\Rec(T_{(N)})$ necessarily dense lineable for every $N \in \NN$?
		
		\item Is the set $\Fc\Rec(T)$ necessarily dense lineable?
		
		\item Is $T\oplus T$ necessarily an $\Fc$-recurrent operator?
	\end{enumerate}
\end{problems}

Note that a positive answer to \textbf{Problem~(a)} would imply an affirmative answer to both \textbf{Problems~(b)} and \textbf{(c)}. However, for usual recurrence (i.e.\ for $\Ic$-recurrence) we already have a negative answer to \textbf{Problems~(a)}~and~\textbf{(c)}:
\begin{enumerate}[--]
	\item {\em in Section~\ref{Sec:3n-qr} we construct recurrent operators $T \in \Lc(X)$ for which $T\oplus T$ is not recurrent}. 
\end{enumerate}
Nevertheless, that example is not solving \textbf{Problem~(b)} in the negative since the mentioned operator fulfills that $\Rec(T)=X$. Thus, finding a possible recurrent operator $T$ for which the set $\Rec(T)$ is not dense lineable demands to solve again the {\em $T\oplus T$-recurrence problem}\footnote{\textbf{Problem~(b)} for both usual and $\AP$-recurrence has recently been solved in the negative; see \cite{LoMe2024_JMAA}.}.\newpage

It is not hard to check that all the operators constructed in Section~\ref{Sec:3n-qr} have a dense lineable set of $\AP$-recurrent vectors. In particular:
\begin{enumerate}[--]
	\item {\em there exist multiple recurrent operators $T \in \Lc(X)$ for which $T\oplus T$ is not even recurrent};
\end{enumerate}
and the same comments done for usual recurrence hold for $\AP$-recurrence. These kind of arguments are not valid for the other recurrence notions considered in this paper:

\begin{proposition}[\textbf{Modification of \cite[Proposition 4]{BMPP2016} and \cite[Proposition 5.6]{BMPP2019}}]\label{Pro:RRec->q-r}
	Let $(X,T)$ be a linear dynamical system. If $T$ is reiteratively recurrent, then $T$ is quasi-rigid, and in particular $\Rec(T_{(N)})$ is dense lineable for every $N \in \NN$.
\end{proposition}
\begin{proof}
	Given a non-empty subset $U \subset X$ we can pick a vector $x \in \RRec(T)\cap U$ and hence $\Bdsup(N_T(x,U))>0$. Recall that
	\[
	\big(N_T(x,U)-N_T(x,U)\big) := \left\{ s_2-s_1 : s_1\leq s_2 \in N_T(x,U) \right\} \subset \left\{ n\geq 0 : T^{n}(U)\cap U\neq\varnothing \right\}.
	\]
	Now we use \cite[Theorem 3.18]{Furstenberg1981_book}, which implies that $A-A \in \Delta^*$ whenever $A \in \BDsup$, obtaining that: for every non-empty open subset $U \subset X$ we have that
	\[
	\{ n\geq0 : T^{n}(U)\cap U\neq\varnothing \} \in \Delta^*.
	\]
	Since $\Del^*$ is a filter (see \cite{BerDown2008,HS1998}) it is easily checked that every $N$-fold direct sum $(X^N,T_{(N)})$ is topologically recurrent. By Theorem~\ref{The:Juan} we have that $T$ is quasi-rigid and Proposition~\ref{Pro:qr->dense.lin} (or Theorem~\ref{The:dense.lineability}) implies that $\Rec(T_{(N)})$ is dense lineable for every $N \in \NN$.
\end{proof}

This last Proposition~\ref{Pro:RRec->q-r} together with the very recent results \cite[Theorems~5.1 and 5.8]{GriLo2023} show that \textbf{Problems~(a)}, \textbf{(b)} and \textbf{(c)} may have a positive answer for Furstenberg families implying stronger recurrence notions than those of usual and $\AP$-recurrence.\\[-5pt]

We would like to mention that, apart from the properties of (dense) lineability, also the {\em spaceability property} (existence of an infinite-dimensional closed subspace) has been studied for the set of recurrent vectors using the notion of {\em quasi-rigidity} in an essential way. Indeed, the results obtained in \cite{Lopez2024_IMRN} justify again that quasi-rigidity is the analogous notion, for recurrence, to that of weak-mixing for hypercyclicity.

\section*{Funding}

The first author was supported by the project FRONT of the French National Research Agency (grant ANR-17-CE40-0021) and by the Labex CEMPI (ANR-11-LABX-0007-01). The second and third authors were supported by Projects PID2019-105011GB-I00 and PID2022-139449NB-I00 from the MCIN/AEI/10.13039/501100011033/FEDER, UE. The second author was also partially supported by the Spanish Ministerio de Ciencia, Innovaci\'on y Universidades, grant FPU2019/04094. The third author was also partially supported by Generalitat Valenciana, Project PROMETEU/2021/070.

\section*{Acknowledgments}

The authors would like to thank Juan B\`es, Antonio Bonilla and Karl Grosse-Erdmann for suggesting the topic, and also for providing valuable ideas on the subject studied. We also want to thank Robert Deville for letting us know about the work \cite{Auge2012}.

\newpage

{\small
$\ $\\

\textsc{Sophie Grivaux}: CNRS, Universit\'e de Lille, UMR 8524 - Laboratoire Paul Painlev\'e, F-59000 Lille, France. e-mail: sophie.grivaux@univ-lille.fr\\

\textsc{Antoni L\'opez-Mart\'inez}: Universitat Polit\`ecnica de Val\`encia, Institut Universitari de Matem\`atica Pura i Aplicada, Edifici 8E, 4a planta, 46022 Val\`encia, Spain. e-mail: alopezmartinez@mat.upv.es\\

\textsc{Alfred Peris}: Universitat Polit\`ecnica de Val\`encia, Institut Universitari de Matem\`atica Pura i Aplicada, Edifici 8E, 4a planta, 46022 Val\`encia, Spain. e-mail: aperis@mat.upv.es
}


{\footnotesize
\begin{thebibliography}{X}
\addcontentsline{toc}{section}{References}
	
	\bibitem{AmBe2019_arXiv} M. Amouch and O. Benchiheb. On recurrent sets of operators. Preprint (2019), arXiv:1907.05930.
	
	\bibitem{Auge2012} J.-M. Aug\'e. Linear Operators with Wild Dynamics. \textit{Proc. Am. Math. Soc.}, vol. \textbf{140}, no. 6, 2103--2116 (2012).
	
	\bibitem{Banks1999} J. Banks. Topological mapping properties defined by digraphs. \textit{Discrete Contin. Dyn. Syst.}, vol. \textbf{5}, no. 1, 83--92 (1999).
	
	\bibitem{BaGri2006} F. Bayart and S. Grivaux. Frequently hypercyclic operators. \textit{Trans. Am. Math. Soc.}, vol. \textbf{358}, no. 11, 5083--5117 (2006).

	\bibitem{BaMa2007nwm} F. Bayart and \'E. Matheron. Hypercyclic operators failing the Hypercyclicity Criterion on classical Banach spaces. \textit{J. Funct. Anal.}, vol. \textbf{250}, 426--441 (2007).

	\bibitem{BaMa2009nwm} F. Bayart and \'E. Matheron. (Non-)Weakly mixing operators and hypercyclicity sets. \textit{Ann. de l'Institut Fourier}, vol. \textbf{59}, no. 1, 1--35 (2009).
	
	\bibitem{BaMa2009_book} F. Bayart and \'E. Matheron. \textit{Dynamics of linear operators}. Cambridge University Press, 2009.
	
	\bibitem{BMPP2016} J. B\`es, Q. Menet, A. Peris, and Y. Puig. Recurrence properties of hypercyclic operators. \textit{Math. Ann.}, vol. \textbf{366}, no. 1-2, 545--572 (2016).
	
	\bibitem{BMPP2019} J. B\`es, Q. Menet, A. Peris, and Y. Puig. Strong transitivity properties for operators. \textit{J. Differ. Equ.}, vol. \textbf{266}, no. 2-3, 1313--1337 (2019).
	
	\bibitem{BesPe1999} J. B\`es and A. Peris. Hereditarily hypercyclic operators. \textit{J. Funct. Anal.}, vol. \textbf{167}, 94--112 (1999).
	
	\bibitem{BerDown2008} V. Bergelson and T. Downarowicz. Large sets of integers and hierarchy of mixing properties of measure preserving systems. \textit{Colloq. Math.}, vol. \textbf{110}, no. 1, 117--150 (2008).
	
	\bibitem{BG2018} A. Bonilla and K.-G. Grosse-Erdmann. Upper frequent hypercyclicity and related notions. \textit{Rev. Mat. Complut.}, vol. \textbf{31}, no. 3, 673--711 (2018).
	
	%
	\bibitem{BoGrLoPe2022} A. Bonilla, K-G. Grosse-Erdmann, A. L\'opez-Mart\'inez, and A. Peris. Frequently recurrent operators. \textit{J. Funct. Anal.}, vol. \textbf{283}, issue 12, no. 109713 (2022).
	
	\bibitem{Bourbaki1989} N. Bourbaki. \textit{General Topology}. Chapters 1-4, Springer, 1989.
	
	\bibitem{CarMur2022_IEOT} R. Cardeccia and S. Muro. Arithmetic progressions and chaos in linear dynamics. \textit{Integral Equ. Oper. Theory}, vol. \textbf{94}, no. 11, 18 pages (2022).
	
	\bibitem{CarMur2022_MS} R. Cardeccia and S. Muro. Multiple recurrence and hypercyclicity. \textit{Math. Scand.}, vol. \textbf{128}, no. 3, 16 pages (2022).
	
	\bibitem{CarMur2022_arXiv} R. Cardeccia and S. Muro. Frequently recurrence properties and block families. \textit{Rev. Real Acad. Cienc. Exactas Fis. Nat. Ser. A-Mat.}, \textbf{119}(60) (2025), 27pages.
	
	\bibitem{CoPa2012} G. Costakis and I. Parissis. Szemer\'edi's theorem, frequent hypercyclicity and multiple recurrence. \textit{Math. Scand.}, vol. \textbf{110}, no. 2, 251--272 (2012).
	
	\bibitem{CoMaPa2014} G. Costakis, A. Manoussos, and I. Parissis. Recurrent linear operators. \textit{Complex Anal. Oper. Theory}, vol. \textbf{8}, 1601--1643 (2014).

	\bibitem{ChenKosVe2021} C.-C. Chen, M. Kosti\'c, and D. Velinov. A note on recurrent strongly continuous semigroups of operators. \textit{Funct. Anal. Approx. Comput.}, vol. \textbf{13}, no. 1, 7--12 (2021).

	\bibitem{DeRead2009} M. De La Rosa and C. Read. A hypercyclic operator whose direct sum $T\oplus T$ is not hypercyclic. \textit{J. Oper. Theory}, vol. \textbf{2}, 369--380 (2009).
	
	\bibitem{EisGri2011} T. Eisner and S. Grivaux. Hilbertian Jamison sequences and rigid dynamical systems. \textit{J. Funct. Anal.}, vol. \textbf{261}, no. 7, 2013--2052 (2011).
	
	\bibitem{ErEsMe2021} R. Ernst, C. Esser, and Q. Menet. $\Uc$-frequent hypercyclicity notions and related weighted densities. \textit{Isr. J. Math.}, vol. \textbf{241}, 817--848 (2021).
	
	\bibitem{Furstenberg1981_book} H. Furstenberg. \textit{Recurrence in Ergodic Theory and Combinatorial Number Theory}. Princeton University Press, 1981.
	
	\bibitem{FursWeiss1977} H. Furstenberg and B. Weiss. \textit{The finite multipliers of infinite ergodic transformations. The structure of attractors in dynamical systems}. Proc. Conf., North Dakota State Univ., Lecture Notes in Math., vol. \textbf{1978}, 127--132 (1977).
	
	\bibitem{GaMaPeOp2015} V. J. Gal\'an, F. Mart\'inez-Jim\'enez, A. Peris, and P. Oprocha. Product Recurrence for Weighted Backward Shifts. \textit{Appl. Math. Inf. Sci.}, vol. \textbf{9}, no. 5, 2361--2365 (2015).
	
	\bibitem{GlasMaon1989} S. Glasner and D. Maon. Rigidity in topological dynamics. \textit{Ergod. Theory Dyn. Syst.}, vol. \textbf{9}, no. 2, 309--320 (1989).

	\bibitem{GTT2010} G. Grekos, V. Toma, and J. Tomanov\'a. A note on uniform or Banach density. \textit{Ann. Math. Blaise Pascal}, vol. \textbf{17}, no. 1, 153--163 (2010).
	
	%
	\bibitem{GriLo2023} S. Grivaux and A. L\'opez-Mart\'inez. Recurrence properties for linear dynamical systems: An approach via invariant measures. \textit{J. Math. Pures Appl.}, vol. \textbf{169}, 155--188 (2023).

	%
	\bibitem{GrivauxLoPe2025_AMP_questions-I} S.~Grivaux, A.~L\'opez-Mart\'inez, and A. Peris. Questions in linear recurrence I: The $T\oplus T$-recurrence problem. \textit{Anal. Math. Phys.}, \textbf{15}(1) (2025), 26 pages.

	%
	\bibitem{GrivauxLoPe2025_BJMA_questions-II} S.~Grivaux, A.~L\'opez-Mart\'inez, and A. Peris. Questions in linear recurrence II: Lineability properties. \textit{Banach J. Math. Anal.}, \textbf{19}(61) (2025), 28 pages.
	
	\bibitem{GrPe2010} K-G. Grosse-Erdmann and A. Peris. Weakly mixing operators on topological vector spaces. \textit{RACSAM}, vol. \textbf{104}, no. 2, 413--426 (2010).
	
	\bibitem{GrPe2011_book} K-G. Grosse-Erdmann and A. Peris. \textit{Linear Chaos}. Springer, 2011.
	
	\bibitem{Herrero1979} D. A. Herrero. Possible structures for the set of cyclic vectors. \textit{Indiana Univ. Math. J.}, vol. \textbf{28}, no. 6, 913--926 (1979).

	\bibitem{Herrero1991} D. A. Herrero. Limits of hypercyclic and supercyclic operators. \textit{J. Funct. Anal.}, vol. \textbf{99}, 179--190 (1991).

	\bibitem{HS1998} N. Hindman and D. Strauss. \textit{Algebra in the Stone-\v{C}ech Compactification}. De Gruyter, 1998.
	
	\bibitem{Jungck1976} G. Jungck. Commuting Mappings and Fixed Points. \textit{Amer. Math. Monthly}, vol. \textbf{83}, no. 4, 261--263 (1976).
	
	\bibitem{KwiLiOpYe2017} D. Kwietniak, J. Li, P. Oprocha, and X. Ye. Multi-recurrence and van der waerden systems. \textit{Science China Mathematics}, vol. \textbf{60}, no. 1, 59--82 (2017).
	
	\bibitem{LindTzaf1977_book} J. Lindestrauss and L. Tzafriri. \textit{Classical Banach spaces I and II}. Springer, 1977.
	
	%
	\bibitem{Lopez2024_IMRN} A. L\'opez-Mart\'inez. Recurrent subspaces in Banach spaces. \textit{Int. Math. Res. Not. IMRN}, vol. \textbf{2024}, no. 11, 9067--9087 (2024).
	
	%
	\bibitem{Lopez2024_RinM} A. L\'opez-Mart\'inez. Invariant measures from locally bounded obits. \textit{Results Math.}, vol. \textbf{79}, no. 185, (2024).
	
	%
	\bibitem{Lopez2025_arXiv_frequently} A.~L\'opez-Mart\'inez. Frequently hypercyclic composition operators on the little Lipschitz space of a rooted tree. \textit{arXiv preprint}, arXiv:2505.02397v1.
	
	%
	\bibitem{LoMe2024_JMAA} A. L\'opez-Mart\'inez and Q. Menet. Two remarks on the set of recurrent vectors. \textit{J. Math. Anal. Appl.}, vol. 541, issue 1, no. 128686 (2024).
	
	\bibitem{Menet2017} Q. Menet. Linear chaos and frequent hypercyclicity. \textit{Trans. Amer. Math. Soc.}, vol. \textbf{369}, 4977--4994 (2017).
	
	\bibitem{MoMuPeSe2022} M. S. Moslehian, G. A. Mu\~noz-Fern\'andez, A. M. Peralta, and J. B. Seoane-Sep\'ulveda. Similarities and differences between real and complex Banach spaces: an overview and recent developments, \textit{RACSAM}, vol. \textbf{116}, no. 2, 80 pages (2022).
	
	\bibitem{MuSaTo1999} G. Mu\~{n}oz-Fern\'andez, Y. Sarantopoulos, and A. Tonge. Complexifications of real Banach spaces, polynomials and multilinear maps. \textit{Studia Math.}, vol. \textbf{134}, no. 1, 1--33 (1999).
	
	\bibitem{OvsPel} R.I. Ovsepian and A. Pe\l czynski. On the existence of a fundamental total and bounded biorthogonal sequence in every separable Banach space, and related constructions of uniformly bounded orthonormal systems in $L^2$. \textit{Studia Math.}, vol. \textbf{54}, 149--159 (1975).
	
	\bibitem{Shkarin2009} S. Shkarin. On the spectrum of frequently hypercyclic operators. \textit{Proc. Am. Math. Soc.}, vol. \textbf{137}, 123--134 (2009).
	
	\bibitem{Szemeredi1975} E. Szemer\'edi. On sets of integers containing no $k$ elements in arithmetic progression. \textit{Proc. of the International
	Congress of Math. (Vancouver, 1974)}, vol. \textbf{2}, 503--505 (1975).

\end{thebibliography}
}
\end{document}